\documentclass[11pt,reqno]{amsart}
\usepackage[margin=1in]{geometry}
\usepackage{dsfont}
\usepackage{amsmath,amssymb,amsfonts,mathtools,bm}
\usepackage{amsthm}
\numberwithin{equation}{section}
\usepackage{enumitem}
\usepackage{booktabs}
\usepackage{algorithm}
\usepackage{algpseudocode}
\usepackage{graphicx}
\usepackage{subcaption}
\usepackage{float}
\usepackage{array}
\usepackage{xcolor}
\usepackage{microtype}

\usepackage{aliascnt}

\usepackage[colorlinks=true, pdfstartview=FitV, linkcolor=cyan,citecolor=red, urlcolor=blue]{hyperref}

\definecolor{labelkey}{rgb}{0,0,1}

\newtheorem{theorem}{Theorem}[section]

\newaliascnt{lemma}{theorem}
\newtheorem{lemma}[lemma]{Lemma}
\aliascntresetthe{lemma}

\newaliascnt{proposition}{theorem}
\newtheorem{proposition}[proposition]{Proposition}
\aliascntresetthe{proposition}

\newaliascnt{corollary}{theorem}
\newtheorem{corollary}[corollary]{Corollary}
\aliascntresetthe{corollary}

\newaliascnt{definition}{theorem}
\newtheorem{definition}[definition]{Definition}
\aliascntresetthe{definition}

\newaliascnt{remark}{theorem}
\newtheorem{remark}[remark]{Remark}
\aliascntresetthe{remark}

\newaliascnt{claim}{theorem}

\aliascntresetthe{claim}

\newaliascnt{assumption}{theorem}
\newtheorem{assumption}[assumption]{Assumption}
\aliascntresetthe{assumption}

\newcommand{\R}{\mathbb R}
\newcommand{\E}{\mathbb E}
\newcommand{\cP}{\mathcal P}
\newcommand{\cQ}{\mathcal Q}
\newcommand{\cW}{\mathcal{W}_2}

\newcommand{\Law}{\operatorname{Law}}
\newcommand{\KL}{D_{\mathrm{KL}}}
\newcommand{\norm}[1]{\left\lVert #1\right\rVert}

\newcommand{\ip}[2]{\left\langle #1,#2\right\rangle}
\newcommand{\dd}{\,\mathrm d}
\newcommand{\barV}{\overline V}
\newcommand{\grad}{\nabla}
\newcommand{\Hess}{\nabla^2}
\newcommand{\defect}{\mathfrak d}

\def \Var {\mathrm{Var}}

\title[Stability of finite-batch particle MFVI]{Stability of Finite-Batch Particle Mean-Field Variational Inference Beyond Strong Convexity}
\author[Vinh Nguyen]{Vinh Nguyen$^1$}
\address{$^1$Department of Mathematics, University of California, Berkeley, CA 94704, USA}
\email{vnguyen26@berkeley.edu}

\author[Truong Vu]{Truong Vu$^{2,3}$}
\address{$^2$Departments of Mathematics and Computational Medicine, University of California, Los Angeles, CA 90095, USA}
\address{$^3$Applied Mathematics and Computational Research (AMCR) division, Lawrence Berkeley National Laboratory, Berkeley, CA 94720, USA}
\email{truongvu@math.ucla.edu}

\date{August 2026}
\thanks{\textbf{Acknowledgment.} V.N. was partially supported by an AMS--Simons travel grant.}

\begin{document}

\begin{abstract}
We study the implementable finite-batch particle algorithm for mean-field
variational inference as a fully discrete stochastic approximation of the
projected Wasserstein dynamics.  The target potential is globally smooth but
need not be strongly convex.  The departure from contractivity is quantified by the curvature defect
\[
\mathfrak d_\alpha(x,y)
=
\bigl[\alpha\|x-y\|^2-
\langle\nabla V(x)-\nabla V(y),x-y\rangle\bigr]_+,
\]
which is the additive loss in the one-step Euler contraction estimate.  We prove a
non-asymptotic Wasserstein stability bound that separates
initialization, product-empirical approximation, finite-batch drift error,
time discretization, and the defects accumulated along the coupled
trajectories.  Under the uniform bound $\mathfrak d_\alpha\leq\beta$, the
particle iterates remain within $O(\sqrt{\beta/\alpha})$ of any MFVI
minimizer, up to explicit errors in the particle number, batch size, and step
size.  The proof uses a stationary comparison array whose population law is
an MFVI minimizer but whose particle-level law is a random product empirical
measure, and it controls the resulting projected-drift discrepancy
explicitly.  We also give coordinatewise defect estimates and structural
conditions for dimension-independent projected-drift sensitivity, construct
an arbitrary-dimensional smooth nonconvex benchmark with a closed-form MFVI
minimizer, and explain why polynomially growing drifts require a modification
of the untamed explicit scheme.
\end{abstract}

\maketitle

\section{Introduction}
\label{sec:introduction}

Mean-field variational inference (MFVI) approximates a probability density
\[
p(x)=Z^{-1}e^{-V(x)},\qquad x\in\R^m,
\]
by minimizing relative entropy over product measures.  The product constraint
makes storage, sampling, and marginal computations tractable and has made
MFVI a standard approximation in Bayesian statistics and machine learning;
see \cite{BKM2017,WJ2008}.  From the computational point of view, however,
the minimization problem is infinite-dimensional, the coordinate marginals
remain coupled through the potential $V$, and an implementable algorithm
introduces additional errors through particle approximation, stochastic
evaluation of the projected drift, and time discretization.  We derive a quantitative stability estimate that keeps these errors separate
for smooth targets that need not be strongly log-concave.

Du et al. \cite{DWZZ2026} introduced the particle algorithm for mean-field
variational inference (PAVI), an interacting-particle discretization of the
projected Wasserstein dynamics associated with MFVI.  Their finite-particle
analysis assumes the global Hessian bounds
\begin{equation*}
\alpha I_m\leq \nabla^2V(x)\leq LI_m,
\qquad x\in\R^m,
\end{equation*}
with $\alpha>0$.  The upper bound controls the explicit Euler step, whereas
the lower bound provides the monotonicity needed for contraction.  We retain
the global smoothness assumption but replace strong convexity by an additive
measure of its failure.  This separates the numerical requirement of a
Lipschitz drift from the geometric mechanism responsible for contraction.

For one Euler step, set
\[
T_h(x)=x-h\nabla V(x)
\]
and every $x,y\in\R^m$,
\begin{equation}
\label{eq:euler-defect-identity}
\begin{aligned}
\|T_h(x)-T_h(y)\|^2
&=
\|x-y\|^2
-2h\langle\nabla V(x)-\nabla V(y),x-y\rangle
+h^2\|\nabla V(x)-\nabla V(y)\|^2
\\
&\leq
(1-2\alpha h+L^2h^2)\|x-y\|^2
+2h\mathfrak d_\alpha(x,y),
\end{aligned}
\end{equation}
where
\begin{equation*}
\mathfrak d_\alpha(x,y)
=
\left[
\alpha\|x-y\|^2
-
\langle\nabla V(x)-\nabla V(y),x-y\rangle
\right]_+.
\end{equation*}
The quantity $\mathfrak d_\alpha$ is the additive loss from the contraction
estimate associated with $\alpha$-strong convexity.  Our main estimate retains
this defect along the coupled trajectories instead of replacing it at once by
a worst-case constant.  The commonly used uniform form
\[
\mathfrak d_\alpha(x,y)\leq\beta
\]
is equivalent to
\[
\langle\nabla V(x)-\nabla V(y),x-y\rangle
\geq
\alpha\|x-y\|^2-\beta,
\]
but the trajectory-dependent formulation also covers situations in which the
nonconvex region is visited only rarely.

Our principal result is a non-asymptotic error estimate for the finite-batch
scheme.  Let $q_{X_n}$ denote the product empirical law of the
PAVI array after $n$ steps, let $q^\star$ be any MFVI minimizer, and let
$\varepsilon_N(q^\star)$ be the root-mean-square product-empirical error of
$N$ independent samples from $q^\star$.  Under the uniform defect bound and
the explicit step-size restriction, \autoref{cor:uniform} gives, up to
universal numerical constants,
\begin{equation}
\label{eq:intro-main-bound}
\begin{aligned}
\bigl(\E \cW^2(q_{X_n},q^\star)\bigr)^{1/2}
&\lesssim
(1-c\alpha h)^{n/2}
\bigl(\E \cW^2(q_{X_0},q^\star)\bigr)^{1/2}
\\
&\quad+
\left(1+\frac{\kappa}{\alpha}\right)
\varepsilon_N(q^\star)
+\sqrt{\frac{\beta}{\alpha}}
+\sqrt{\frac{h\Gamma_\star}{\alpha B}}
+\frac{L}{\alpha}\sqrt{h(h\Gamma_\star+m)},
\end{aligned}
\end{equation}
where
\[
\Gamma_\star=\E_{q^\star}\|\nabla V\|^2
\]
and $\kappa$ measures the sensitivity of the projected drift to perturbations
of the product law.  The sharper estimate in \autoref{thm:main-defect}
replaces $\beta$ by a geometrically weighted sum of the defects encountered by
the stationary coupling.

Estimate \eqref{eq:intro-main-bound} separates the computational parameters
of the method.  The particle number $N$ controls the product-empirical error,
the batch size $B$ controls the stochastic projected-drift error, and the
step size $h$ controls both the transient contraction and the Euler error.  A
generic implementation requires
\[
O(mNB\,C_{\partial V})
\]
work per iteration when one coordinate derivative costs $C_{\partial V}$;
see \autoref{prop:complexity}.  Over a fixed physical time horizon $T$, this
becomes $O(mNBT\,C_{\partial V}/h)$.  Consequently, increasing $B$ is useful
only until the batch error is comparable with the particle and discretization
errors.  Analyzing the finite-batch algorithm directly exposes this accuracy--cost
tradeoff, which is hidden after passing to the full projected drift or the
continuum dynamics.

The proof uses a comparison process that is stationary at the population level
and compatible with the empirical algorithm.  We construct an array
whose rows have the marginals of $q^\star$ and couple it to the PAVI array by
predictable rowwise optimal matchings.  The product of the stationary row
empirical measures is random and is not equal to the deterministic product
law $q^\star$.  Hence the projected drift evaluated at this empirical law has
a nonvanishing discrepancy from the drift at $q^\star$, even after
conditioning.  This discrepancy produces the $\kappa\varepsilon_N(q^\star)$ term in
\eqref{eq:intro-main-bound}.  The same
one-step recursion then incorporates the finite-batch variance, the local
Euler--Maruyama remainder, and the encountered curvature defect.  Random
subsampling is related in spirit to random-batch methods for interacting
particle systems \cite{JLL2020}, but here the batch approximates a
coordinatewise expectation under a random product empirical law, so a
separate law-approximation estimate is unavoidable.

Without strong convexity, the analysis also requires several variational and
integrability estimates.  Under the uniform defect condition we prove quadratic coercivity
of $V$ and existence of MFVI minimizers, establish the implication from global
minimizers to the coordinatewise Gibbs equations, derive centered moment
bounds for the stationary marginals without log-concavity, and prove
well-posedness of the independent-projection McKean--Vlasov diffusion.  The
continuous-time coupling satisfies an analogous accumulated-defect estimate.
We also introduce coordinatewise defects and structural bounds for $\kappa$.
Although the generic estimate is $\kappa\leq\sqrt m\,L$, sparse or bounded
cross-coordinate interactions can give a dimension-independent bound.  Since
squared Wasserstein distance is additive across product coordinates,
$m^{-1/2}\cW$ is the per-coordinate root-mean-square metric used to interpret
dimension dependence.

The present analysis is connected with several strands of work on MFVI,
Wasserstein dynamics, and particle computation.  Convergence of coordinate
ascent variational inference has been studied under log-concavity and related
conditions in \cite{AL2026,BPY2025,LZ2024}; Caprio, Corenflos, and Power
\cite{CCP2026} obtain local Wasserstein contraction under a
transport-information inequality and functional smoothness assumptions.  The
Wasserstein-gradient-flow formulation of variational inference builds on
\cite{AGS2008,JKO1998}; independent projections and quantitative mean-field
approximations are developed in \cite{L2026,LMY2024,YY2025}.  Jiang, Chewi,
and Pooladian \cite{JCP2026} give a polyhedral-optimization approach to MFVI
in Wasserstein space.  More broadly, computational discretizations of
Wasserstein gradient flows include primal-dual schemes \cite{CCWW2022}, while
regularized Stein variational gradient flow provides a particle-based
approximation of KL gradient flow with convergence, discretization,
well-posedness, and stability analysis \cite{HBSL2025}.

Other recent variational methods are complementary to PAVI.  Wu and Blei
\cite{WB2026} extend the mean-field family through entropic regularization,
Lyu et al. \cite{LYMS2026} develop a mini-batch primal-dual method for
large-scale parametric MFVI, generalized Wasserstein particle flows are
studied in \cite{CZYZ2023}, and non-asymptotic bounds for particle gradient
descent appear in \cite{CKPJ2025}.  Borghi and Carrillo \cite{BC2026} use
Gaussian interacting particles in the Bures--Wasserstein geometry.  These
methods differ from PAVI in the variational class, parameterization, or
measure-valued dynamics.  Perturbation stability of the MFVI optimizer itself
was studied for strongly log-concave targets in \cite{SWGN2025}; our target
and optimizer are fixed, and the stability problem concerns the fully
discrete particle dynamics.

The curvature-defect viewpoint is also related to coupling methods for
nonconvex Langevin dynamics under dissipativity or curvature-at-infinity
conditions \cite{E2016,EGZ2019,MMS2020}.  Such methods can recover strict contraction in a modified transport distance
for the exact diffusion.  PAVI
requires, in addition, control of the coordinatewise law dependence, random
product empirical approximation, and finite-batch drift evaluation.  At a
broader level, interacting swarms have been analyzed for finite-dimensional
nonconvex optimization in \cite{LTZ2024,TZ2024,DGLT2024}.  Those algorithms use communication among agents to locate minimizers of a
finite-dimensional objective, whereas PAVI uses particles to approximate a
product probability law; their connection here is therefore at the level of
interacting-particle optimization rather than MFVI itself.

The assumption $\|\nabla^2V\|_{\mathrm{op}}\leq L$ is retained throughout.
For the untamed explicit algorithm this is a numerical restriction, not merely
an analytical hypothesis: explicit Euler--Maruyama schemes can
diverge for superlinearly growing drifts even under favorable coercivity
conditions \cite{HJK2011,HJK2012}, while tamed or otherwise modified schemes
are designed for that regime \cite{BDMS2019}.  We therefore isolate the
polynomial-growth case as an obstruction to the present algorithm rather than
claiming an extension that would require a different discretization.

We also construct an arbitrary-dimensional smooth nonconvex interaction model
with a closed-form unique MFVI minimizer.  The
benchmark permits direct measurement of particle, step-size, batch, defect,
and dimension effects without an unknown optimization error in the reference
law.  The numerical experiments are intended to test the scaling mechanisms
in the theorem rather than to serve as an empirical comparison among
variational-inference algorithms.

The remainder of the paper is organized as follows.
\autoref{sec:mfvi-defects} introduces the MFVI problem and the curvature-defect
framework.  \autoref{sec:variational-foundations} establishes the variational
foundations and stationary moment estimates.  \autoref{sec:independent-projection}
studies the independent-projection diffusion and its continuous-time defect
estimate.  \autoref{sec:finite-batch-pavi} presents finite-batch PAVI and its
computational cost.  \autoref{sec:stationary-array} develops the
stationary-array coupling and the one-step recursion, and
\autoref{sec:main-stability} proves the non-asymptotic stability bounds.
\autoref{sec:benchmark} gives the nonconvex benchmark,
\autoref{sec:numerics} reports the numerical experiments,
\autoref{sec:polynomial} explains the polynomial-drift obstruction, and
\autoref{sec:discussion} concludes with limitations and extensions.
\section{MFVI and curvature defects}
\label{sec:mfvi-defects}

For $p\geq 1$, let $\cP_p(\R^d)$ denote the set of Borel
probability measures on $\R^d$ with finite $p$-th moment, and let
$\cP_{\mathrm{ac}}(\R^d)$ denote the set of Borel probability measures
that are absolutely continuous with respect to Lebesgue measure. We write
\[
\cP_{p,\mathrm{ac}}(\R^d)
:=
\cP_p(\R^d)\cap\cP_{\mathrm{ac}}(\R^d).
\]
When no confusion can arise, we use the same symbol for an
absolutely continuous probability measure and its density with
respect to Lebesgue measure.  We adopt the convention $0\log0=0$. For $\mu,\nu\in\cP_2(\R^d)$, the quadratic Wasserstein distance is
\[
\cW^2(\mu,\nu)
=
\inf_{\gamma\in\Pi(\mu,\nu)}
\int_{\R^d\times\R^d}
\norm{x-y}^2\,\gamma(\dd x,\dd y),
\]
where $\Pi(\mu,\nu)$ denotes the set of couplings of $\mu$ and $\nu$
(see, for example, \cite{V2009}).

\smallskip
The mean-field family on $\R^m$ is
\[
\cQ
=
\left\{
q^1\otimes\cdots\otimes q^m:
q^i\in\cP_{\mathrm{ac}}(\R),\ i\in[m]
\right\},
\qquad
[m]:=\{1,\ldots,m\}.
\]
Given a potential $V:\R^m\to\R$, set
\[
Z:=\int_{\R^m}e^{-V(x)}\,\dd x.
\]
Whenever $Z<\infty$, the associated target probability density is
\[
p(x)=Z^{-1}e^{-V(x)}.
\]
The mean-field variational inference problem is
\begin{equation}
\label{eq:mfvi}
\inf_{q\in\cQ}\KL(q\|p).
\end{equation}
Under the uniform curvature-defect condition introduced below, quadratic
coercivity in \autoref{lem:coercivity} implies $Z<\infty$. Hence the main
results require no separate normalization assumption.

\smallskip
Throughout the numerical analysis we impose the following global smoothness
assumption.

\begin{assumption}
\label{ass:smooth}
The potential $V\in C^2(\R^m)$, and there exists $L>0$ such that
\begin{equation*}
\norm{\Hess V(x)}_{\mathrm{op}}\leq L,
\qquad x\in\R^m.
\end{equation*}
\end{assumption}

The gradient $\grad V$ is therefore globally $L$-Lipschitz, and Taylor's
formula gives
\begin{equation}
\label{eq:quadratic-growth-V}
\left|
V(x)-V(0)-\ip{\grad V(0)}{x}
\right|
\leq
\frac{L}{2}\norm{x}^2,
\qquad x\in\R^m,
\end{equation}
so $V$ has at most quadratic growth in absolute value.

For $i\in[m]$, write $x^{-i}$ for the collection of all coordinates of
$x$ except the $i$-th one. If $q=\bigotimes\limits_{i=1}^m q^i$, define
\[
q^{-i}:=\bigotimes_{k\neq i}q^k.
\]
For $z^{-i}\in\R^{m-1}$ and $x\in\R$, let
\[
z[i\leftarrow x]
=
(z_1,\ldots,z_{i-1},x,z_{i+1},\ldots,z_m).
\]
For $\mu\in\cP_2(\R^{m-1})$, define the projected potential
\begin{equation}
\label{eq:barV}
\barV_i(x,\mu)
=
\int_{\R^{m-1}}
V\bigl(z[i\leftarrow x]\bigr)\,\mu(\dd z^{-i}).
\end{equation}
The integral in \eqref{eq:barV} is finite by
\eqref{eq:quadratic-growth-V} and the second-moment assumption on $\mu$.
The projected potentials enter both the coordinatewise Gibbs updates and the
drift of the particle dynamics.

\smallskip
For $r\in\R$, write $[r]_+=\max\{r,0\}$.  The following quantity measures
the failure of strong monotonicity.

\begin{definition}
\label{def:defect}
For $\alpha>0$, the curvature defect of $V$ is
\begin{equation*}
\defect_\alpha(x,y)
=
\left[
\alpha\norm{x-y}^2
-
\ip{\grad V(x)-\grad V(y)}{x-y}
\right]_+,
\qquad x,y\in\R^m.
\end{equation*}
For $\beta\geq 0$, we say that $V$ has uniform
$(\alpha,\beta)$-defect if
\begin{equation}
\label{eq:uniform-defect}
\defect_\alpha(x,y)\leq\beta,
\qquad x,y\in\R^m.
\end{equation}
Equivalently,
\begin{equation}
\label{eq:defective-monotonicity}
\ip{\grad V(x)-\grad V(y)}{x-y}
\geq
\alpha\norm{x-y}^2-\beta,
\qquad x,y\in\R^m.
\end{equation}
\end{definition}

When $\beta=0$, condition \eqref{eq:defective-monotonicity} is the usual
$\alpha$-strong monotonicity of $\grad V$. For $V\in C^2$, it is
equivalent to
\[
\Hess V(x)\geq\alpha I_m,
\qquad x\in\R^m.
\]
The parameter $\beta$ therefore measures the additive failure of global
strong convexity. Under \autoref{ass:smooth}, a finite uniform
$(\alpha,\beta)$-defect also requires $\alpha\leq L$, since
\[
\ip{\grad V(x)-\grad V(y)}{x-y}
\leq
L\norm{x-y}^2,
\]
and comparison with \eqref{eq:defective-monotonicity} at arbitrarily large
separations yields the claim.

The product structure also leads to a coordinatewise form of the defect.

\begin{remark}
\label{rem:coordinate-defects}
The product structure permits coordinatewise bookkeeping when nonconvexity is
localized to only some coordinates or interactions.  Define
\[
\defect_{\alpha,i}(x,y)
=
\left[
\alpha|x_i-y_i|^2
-
\bigl(\partial_iV(x)-\partial_iV(y)\bigr)(x_i-y_i)
\right]_+.
\]
Since $[\sum_i a_i]_+\leq\sum_i[a_i]_+$,
\begin{equation*}
\defect_\alpha(x,y)
\leq
\sum_{i=1}^m\defect_{\alpha,i}(x,y).
\end{equation*}
Consequently, if
$\defect_{\alpha,i}(x,y)\leq\beta_i$ for all $x,y$, then every theorem below
that uses the uniform full-dimensional bound remains valid with
$\beta=\sum_i\beta_i$.  In particular, after normalizing the product metric by
\[
\overline{\cW}(\mu,\nu)=m^{-1/2}\cW(\mu,\nu),
\]
the curvature residual in \autoref{cor:uniform} is controlled by the average
coordinate defect,
\[
\frac1{\sqrt m}\sqrt{\frac{12\beta}{\alpha}}
=
\sqrt{\frac{12}{\alpha m}\sum_{i=1}^m\beta_i}.
\]
Hence an $O(\sqrt m)$ bound in the unnormalized product Wasserstein distance
reflects the usual product scaling when each coordinate contributes $O(1)$;
it need not represent a deterioration of the per-coordinate error.  The coordinatewise formulation is sharper when only a sparse collection of
coordinates or interactions contributes appreciably to the defect. Strong
cross-coordinate interactions, however, may make the individual coordinate
defects large.
\end{remark}

Bounded-gradient perturbations of strongly convex potentials provide one class
with finite defect. Such perturbations may remain nonconvex at
arbitrarily large distances and therefore need not satisfy any
convexity at infinity condition.

\begin{proposition}
\label{prop:bounded-perturbation}
Let
\[
V=V_0+U,
\]
where $V_0\in C^2(\R^m)$, $U\in C^1(\R^m)$, and suppose that
\[
\Hess V_0(x)\geq\alpha_0I_m,
\qquad x\in\R^m,
\qquad
\sup_{x\in\R^m}\norm{\grad U(x)}\leq G
\]
for some $\alpha_0>0$ and $G\geq0$. Then, for every
$\alpha\in(0,\alpha_0)$, the potential $V$ has uniform
$(\alpha,\beta)$-defect with
\begin{equation*}
\beta
=
\frac{G^2}{\alpha_0-\alpha}.
\end{equation*}
\end{proposition}

\begin{proof}
Fix $x,y\in\R^m$ and set $r=\norm{x-y}$. Strong convexity of $V_0$
gives
\[
\ip{\grad V_0(x)-\grad V_0(y)}{x-y}
\geq
\alpha_0r^2.
\]
On the other hand,
\[
\begin{aligned}
\ip{\grad U(x)-\grad U(y)}{x-y}
&\geq
-\norm{\grad U(x)-\grad U(y)}\,r\\
&\geq
-2Gr.
\end{aligned}
\]
Consequently,
\[
\ip{\grad V(x)-\grad V(y)}{x-y}
\geq
\alpha_0r^2-2Gr.
\]
Since
\[
2Gr
\leq
(\alpha_0-\alpha)r^2
+
\frac{G^2}{\alpha_0-\alpha},
\]
we obtain
\[
\ip{\grad V(x)-\grad V(y)}{x-y}
\geq
\alpha r^2
-
\frac{G^2}{\alpha_0-\alpha}.
\]
Hence \eqref{eq:uniform-defect} holds.
\end{proof}

\begin{remark}
\label{rem:alpha-choice}
The choice of $\alpha$ balances the contractive part of the estimate against
the additive defect. In particular,
\[
\frac{\beta}{\alpha}
=
\frac{G^2}{\alpha(\alpha_0-\alpha)}
\]
is minimized at $\alpha=\alpha_0/2$. For this choice,
\[
\beta=\frac{2G^2}{\alpha_0},
\qquad
\frac{\beta}{\alpha}
=
\frac{4G^2}{\alpha_0^2}.
\]
This choice minimizes the leading residual $\sqrt{\beta/\alpha}$. Other
values of $\alpha$ may nevertheless give a better balance with the
contraction rate and the remaining error terms.
\end{remark}

A second class consists of potentials that are strongly convex outside a
bounded region but may have negative curvature inside it.
Localized nonconvexity of this type still yields a finite uniform defect.

\begin{proposition}
\label{prop:outside}
Suppose that $V\in C^2(\R^m)$ and that, for some
$\alpha_0>0$, $M\geq0$, and $R>0$,
\[
\Hess V(x)\geq\alpha_0I_m
\quad\text{ whenever }\norm{x}\geq R,
\qquad
\Hess V(x)\geq-MI_m
\quad\text{ for every }x\in\R^m.
\]
Then, for every $\alpha\in(0,\alpha_0)$, the potential $V$ has uniform
$(\alpha,\beta)$-defect with
\begin{equation}
\label{eq:beta-outside-ball}
\beta
=
\frac{(\alpha_0+M)^2R^2}{\alpha_0-\alpha}.
\end{equation}
In particular, choosing $\alpha=\alpha_0/2$ gives
\[
\beta
=
\frac{2(\alpha_0+M)^2R^2}{\alpha_0}.
\]
\end{proposition}

\begin{proof}
For $x=y$ there is nothing to prove.  Assume $x\neq y$ and set
\[
z_t=y+t(x-y),
\qquad
S=\big\{t\in[0,1]:\norm{z_t}<R\big\}.
\]
Because the intersection of a line with the ball $B_R(0)$ has length at
most $2R$, and the curve $t\mapsto z_t$ has constant speed
$\norm{x-y}$, we have
\begin{equation}
\label{eq:S-length}
|S|\,\norm{x-y}\leq 2R.
\end{equation}
The fundamental theorem of calculus gives
\[
\begin{aligned}
\ip{\grad V(x)-\grad V(y)}{x-y}
&=
\int_0^1
\ip{\Hess V(z_t)(x-y)}{x-y}\,\dd t\\
&\geq
\alpha_0(1-|S|)\norm{x-y}^2
-
M|S|\norm{x-y}^2\\
&=
\alpha_0\norm{x-y}^2
-
(\alpha_0+M)|S|\norm{x-y}^2.
\end{aligned}
\]
Using \eqref{eq:S-length}, we obtain
\[
\ip{\grad V(x)-\grad V(y)}{x-y}
\geq
\alpha_0\norm{x-y}^2
-
2(\alpha_0+M)R\norm{x-y}.
\]
Let
\[
C=(\alpha_0+M)R.
\]
Young's inequality yields
\[
2C\norm{x-y}
\leq
(\alpha_0-\alpha)\norm{x-y}^2
+
\frac{C^2}{\alpha_0-\alpha}.
\]
Therefore,
\[
\ip{\grad V(x)-\grad V(y)}{x-y}
\geq
\alpha\norm{x-y}^2
-
\frac{(\alpha_0+M)^2R^2}{\alpha_0-\alpha},
\]
which is \eqref{eq:beta-outside-ball}.
\end{proof}

The two criteria cover different nonconvex regimes. For example,
\[
V_0(x)=\frac{\alpha_0}{2}\norm{x}^2,
\qquad
U(x)=a\sum_{i=1}^m\cos x_i,
\]
satisfies the assumptions of \autoref{prop:bounded-perturbation}, with
$G\leq a\sqrt m$. If $a>\alpha_0$, the resulting potential is nonconvex
at points of arbitrarily large norm and hence is not convex outside any
compact set. By contrast, \autoref{prop:outside} permits arbitrary bounded
negative curvature on a fixed ball, provided strong convexity is recovered
outside that ball. Bounded-gradient perturbations and localized nonconvexity therefore give two
distinct mechanisms for obtaining a finite curvature defect.

A symmetric Gaussian mixture provides a familiar statistical example.

\begin{remark}
\label{rem:gaussian-mixture}
The class also contains a standard smooth multimodal model.  If
$a\in\R^m$ and
\[
p_a(x)=\frac12\varphi_m(x-a)+\frac12\varphi_m(x+a),
\]
where $\varphi_m$ is the standard Gaussian density, then, up to an additive
constant,
\[
V_a(x)=\frac12\|x\|^2-\log\cosh(a^{\mathsf T}x).
\]
Writing $V_a=\|x\|^2/2+U_a$, one has
$\|\nabla U_a(x)\|\leq\|a\|$ and
$\|\nabla^2V_a(x)\|_{\mathrm{op}}\leq1+\|a\|^2$.  Hence, for every
$\alpha\in(0,1)$, \autoref{prop:bounded-perturbation} gives the uniform
defect bound
\[
\beta=\frac{\|a\|^2}{1-\alpha}.
\]
Thus the assumptions include standard non-log-concave mixture targets. The
numerical section instead uses a benchmark with an explicit MFVI minimizer so
that optimization error in the reference solution does not obscure the
particle error.
\end{remark}

A finite uniform defect does not by itself imply uniqueness, as the following
example shows.

\begin{proposition}
\label{prop:multiple-minimizers}
Let $m=2$, let $\gamma$ denote the standard Gaussian law on $\R$, and set
\[
V_J(x_1,x_2)
=
\frac12(x_1^2+x_2^2)-J\tanh(x_1)\tanh(x_2),
\qquad J>0.
\]
Define
\[
\sigma^2=\int_\R \tanh^2(x)\,\gamma(\dd x)>0.
\]
For every $\alpha\in(0,1)$, $V_J$ is globally smooth and has uniform
$(\alpha,\beta)$-defect with
\[
\beta=\frac{2J^2}{1-\alpha}.
\]
If $J>\sigma^{-2}$, the associated MFVI problem has at least two distinct
global minimizers.
\end{proposition}

\begin{proof}
Write $V_J=V_0+U$ with $V_0(x)=\|x\|^2/2$.  Since
\[
\nabla U(x_1,x_2)
=
-J\bigl(\operatorname{sech}^2(x_1)\tanh(x_2),
\tanh(x_1)\operatorname{sech}^2(x_2)\bigr),
\]
we have $\sup_x\|\nabla U(x)\|\leq\sqrt2J$.  The defect bound follows from \autoref{prop:bounded-perturbation}.  Since
$\tanh$ and its derivatives are bounded, the Hessian of $V_J$ is globally
bounded as well.
Existence of an MFVI minimizer follows from \autoref{thm:existence}.

For a product law $q=q^1\otimes q^2$, set
$r_i=\int\tanh x\,q^i(\dd x)$.  Up to an additive constant independent of
$q$, the MFVI objective is
\begin{equation}
\label{eq:multiple-minimizer-objective}
\KL(q^1\|\gamma)+\KL(q^2\|\gamma)-Jr_1r_2.
\end{equation}
Let
\[
\psi(t)=\log\int_\R e^{t\tanh x}\,\gamma(\dd x),
\qquad
q_t(\dd x)=e^{t\tanh x-\psi(t)}\gamma(\dd x).
\]
By symmetry, $\psi'(0)=0$ and $\psi''(0)=\sigma^2$.  Hence
\[
\psi(t)=\frac{\sigma^2}{2}t^2+o(t^2),
\qquad
\psi'(t)=\sigma^2t+o(t).
\]
Moreover,
\[
\KL(q_t\|\gamma)=t\psi'(t)-\psi(t),
\qquad
\int\tanh x\,q_t(\dd x)=\psi'(t).
\]
Evaluating \eqref{eq:multiple-minimizer-objective} at
$q_t\otimes q_t$ gives, relative to $\gamma\otimes\gamma$,
\[
2\bigl(t\psi'(t)-\psi(t)\bigr)-J\psi'(t)^2
=
\sigma^2(1-J\sigma^2)t^2+o(t^2).
\]
If $J>\sigma^{-2}$, this is negative for all sufficiently small nonzero
$t$.  Thus $\gamma\otimes\gamma$ is not a global minimizer and the minimum
value is strictly negative.

Finally, the objective is invariant under simultaneous reflection
$q^i\mapsto(-\mathrm{id})_\#q^i$.  Any reflection-invariant product law has
$r_1=r_2=0$ and therefore has nonnegative objective gap, with equality only
at $\gamma\otimes\gamma$.  Hence a global minimizer of the strictly negative
minimum cannot be reflection invariant; reflecting it produces a second,
distinct global minimizer.
\end{proof}

\smallskip
Thus global smoothness controls the discretization, whereas the curvature
defect controls the loss of contractivity.  The next section establishes the
variational and moment estimates used in the particle analysis.

%%%
\section{Variational foundations and stationary moments}
\label{sec:variational-foundations}

Throughout this section we impose \autoref{ass:smooth} and the uniform defect
condition \eqref{eq:uniform-defect}.  The defect yields quadratic confinement,
normalization of the target, and compactness of minimizing sequences.  We then
derive the coordinatewise Gibbs identities and the stationary moment and
empirical-measure estimates used in the continuous and particle analyses.

The defect assumption first yields the coercivity needed for the direct
method.

\begin{lemma}
\label{lem:coercivity}
Assume that $V$ has uniform $(\alpha,\beta)$-defect. Then there exists
$C_V<\infty$ such that
\begin{equation}
\label{eq:coercivity}
V(x)\geq \frac{\alpha}{8}\norm{x}^2-C_V,
\qquad x\in\R^m.
\end{equation}
Consequently,
\[
Z=\int_{\R^m}e^{-V(x)}\,\dd x<\infty,
\]
and, for every $\eta\in(0,\alpha/8)$,
\begin{equation}
\label{eq:exp-moment-p}
\int_{\R^m}e^{\eta\norm{x}^2}\,p(\dd x)<\infty.
\end{equation}
\end{lemma}

\begin{proof}
Let
\[
G_0=\norm{\grad V(0)}.
\]
For $u\in\mathbb S^{m-1}$ and $r>0$, the defect inequality applied to
$ru$ and $0$ gives
\[
\ip{\grad V(ru)-\grad V(0)}{ru}
\geq
\alpha r^2-\beta.
\]
Dividing by $r$ and using
$\ip{\grad V(0)}{u}\geq-G_0$, we obtain
\[
\frac{\dd}{\dd r}V(ru)
=
\ip{\grad V(ru)}{u}
\geq
\alpha r-\frac{\beta}{r}-G_0.
\]
Set
\[
R
:=
\max\left\{
1,\sqrt{\frac{2\beta}{\alpha}},
\frac{4G_0}{\alpha}
\right\}.
\]
For $r\geq R$,
\[
\frac{\beta}{r}\leq\frac{\alpha r}{2},
\qquad
G_0\leq\frac{\alpha r}{4},
\]
and hence
\[
\frac{\dd}{\dd r}V(ru)\geq\frac{\alpha r}{4}.
\]
Integrating from $R$ to $r\geq R$ yields
\[
V(ru)
\geq
V(Ru)+\frac{\alpha}{8}(r^2-R^2).
\]
Define
\[
m_R:=\min_{\norm{x}\leq R}V(x),
\qquad
s_R:=\min_{\norm{u}=1}V(Ru).
\]
Both constants are finite by continuity. If $\norm{x}\geq R$, the preceding
estimate gives
\[
V(x)
\geq
\frac{\alpha}{8}\norm{x}^2
+s_R-\frac{\alpha R^2}{8},
\]
whereas, if $\norm{x}\leq R$,
\[
V(x)
\geq
m_R
\geq
\frac{\alpha}{8}\norm{x}^2
-
\left(\frac{\alpha R^2}{8}-m_R\right).
\]
Thus \eqref{eq:coercivity} holds, for example, with
\[
C_V
=
\max\left\{
0,\frac{\alpha R^2}{8}-m_R,
\frac{\alpha R^2}{8}-s_R
\right\}.
\]
The estimate \eqref{eq:coercivity} implies
\[
e^{-V(x)}
\leq
e^{C_V}e^{-\alpha\norm{x}^2/8},
\]
so $Z<\infty$. Similarly, for every $\eta<\alpha/8$,
\[
e^{\eta\norm{x}^2-V(x)}
\leq
e^{C_V}
e^{-(\alpha/8-\eta)\norm{x}^2},
\]
which gives \eqref{eq:exp-moment-p}.
\end{proof}
The direct method for \eqref{eq:mfvi} also requires that independence be
preserved under weak convergence. The following elementary closure property
will be used repeatedly.

\begin{lemma}
\label{lem:product-closed}
Let
\[
q_n=q_n^1\otimes\cdots\otimes q_n^m\in\mathcal P(\R^m),
\qquad
q_n\Rightarrow q.
\]
Then
\[
q_n^i\Rightarrow q^i,
\qquad i\in[m],
\]
where $q^i$ is the $i$-th marginal of $q$, and
\[
q=q^1\otimes\cdots\otimes q^m.
\]
In particular, the family of product probability measures is weakly closed.
\end{lemma}

\begin{proof}
The convergence of the marginals follows from the continuous mapping theorem
applied to the coordinate projections. Let $\widehat q_n$, $\widehat q$,
and $\widehat q^{\,i}$ denote the corresponding characteristic functions.
For every $t=(t_1,\ldots,t_m)\in\R^m$,
\[
\widehat q(t)
=
\lim_{n\to\infty}\widehat q_n(t)
=
\lim_{n\to\infty}
\prod_{i=1}^m\widehat q_n^{\,i}(t_i)
=
\prod_{i=1}^m\widehat q^{\,i}(t_i).
\]
The last expression is the characteristic function of
$q^1\otimes\cdots\otimes q^m$. Uniqueness of characteristic functions
therefore gives the conclusion.
\end{proof}

We next prove existence of an MFVI minimizer and the entropy properties needed
for coordinatewise variations.

\begin{theorem}
\label{thm:existence}
Under \autoref{ass:smooth} and the uniform defect condition
\eqref{eq:uniform-defect}, problem \eqref{eq:mfvi} admits a minimizer
$q^\star\in\cQ\cap\cP_2(\R^m)$.
\end{theorem}

\begin{proof}
Global smoothness gives
\[
\left|
V(x)-V(0)-\ip{\grad V(0)}{x}
\right|
\leq
\frac{L}{2}\norm{x}^2,
\]
and hence
\[
|V(x)|\leq C(1+\norm{x}^2)
\]
for some $C<\infty$. Together with \autoref{lem:coercivity}, this implies that every nondegenerate
product Gaussian has finite relative entropy with respect to $p$, so the
infimum in \eqref{eq:mfvi} is finite.

Let $(q_n)\subset\cQ$ be a minimizing sequence. After discarding finitely
many terms, we may assume that
\[
\KL(q_n\|p)\leq K
\]
for some $K<\infty$. Fix $\eta\in(0,\alpha/8)$. The entropy inequality
\[
\int_{\R^m}\phi\,\dd q
\leq
\KL(q\|p)
+
\log\int_{\R^m}e^\phi\,\dd p
\]
applied with $\phi(x)=\eta\norm{x}^2$, together with
\eqref{eq:exp-moment-p}, yields
\[
\sup_n\int_{\R^m}\norm{x}^2\,q_n(\dd x)<\infty.
\]
Thus $(q_n)$ is tight. By Prokhorov's theorem, after passing to a
subsequence,
\[
q_n\Rightarrow q^\star.
\]
By \autoref{lem:product-closed}, $q^\star$ is a product probability
measure. Moreover, relative entropy with respect to the fixed probability
measure $p$ is weakly lower semicontinuous, so
\[
\KL(q^\star\|p)
\leq
\liminf_{n\to\infty}\KL(q_n\|p)
=
\inf_{q\in\cQ}\KL(q\|p)
<\infty.
\]
In particular, $q^\star\ll p$. Since $p$ has a strictly positive
Lebesgue density, $q^\star\ll\mathcal L^m$.

Write
\[
q^\star=q^{\star,1}\otimes\cdots\otimes q^{\star,m}.
\]
If $A\subset\R$ has Lebesgue measure zero, then
$A\times[-R,R]^{m-1}$ has $m$-dimensional Lebesgue measure zero for every
$R>0$. Hence
\[
q^\star\bigl(A\times[-R,R]^{m-1}\bigr)=0.
\]
Letting $R\to\infty$ and using monotone convergence gives
$q^{\star,i}(A)=0$. Thus every marginal is absolutely continuous and
$q^\star\in\cQ$. The uniform second-moment estimate and weak lower
semicontinuity of $x\mapsto\norm{x}^2$ also imply that
$q^\star\in\cP_2(\R^m)$. Hence $q^\star$ is an MFVI minimizer.

We also need finiteness of the marginal entropy integrals. Let
$f^\star$ be the joint density of $q^\star$. For every $c>0$,
\[
t\log t+c\norm{x}^2t
\geq
-e^{-1-c\norm{x}^2},
\qquad t\geq0.
\]
Since $q^\star\in\cP_2(\R^m)$, the negative part of
$f^\star\log f^\star$ is integrable. The entropy integral is therefore
well defined, and
\[
\int_{\R^m}f^\star\log f^\star\,\dd x
=
\KL(q^\star\|p)
-
\int_{\R^m}V\,\dd q^\star
-
\log Z
\]
is finite.

Let $f_i$ be the density of $q^{\star,i}$. Applying the scalar inequality
\[
t\log t+cx^2t\geq-e^{-1-cx^2}
\]
shows that the negative part of $f_i\log f_i$ is integrable. Since
$f^\star(x)=\prod_{i=1}^m f_i(x_i)$, Fubini's theorem gives
\[
\int_{\R^m}f^\star\log f^\star\,\dd x
=
\sum_{i=1}^m\int_{\R}f_i\log f_i\,\dd x.
\]
Each marginal integral has finite negative part, and their sum is finite; hence
each marginal entropy is finite.
\end{proof}

Averaging over all but one coordinate preserves the regularity and defective
monotonicity needed below.

\begin{lemma}
\label{lem:conditional}
For every $i\in[m]$, every $\mu\in\cP_2(\R^{m-1})$, and all
$x,y\in\R$,
\begin{align}
\bigl(\barV_i'(x,\mu)-\barV_i'(y,\mu)\bigr)(x-y)
&\geq
\alpha|x-y|^2-\beta,
\label{eq:conditional-defect}\\
|\barV_i''(x,\mu)|
&\leq L.
\label{eq:conditional-smooth}
\end{align}
Moreover, $f=\barV_i(\,\cdot\,,\mu)$ satisfies
\begin{equation}
\label{eq:conditional-coercivity}
f(x)\geq\frac{\alpha}{8}x^2-C_{i,\mu},
\qquad
|f'(x)|\leq |f'(0)|+L|x|
\end{equation}
for some finite constant $C_{i,\mu}$.
\end{lemma}

\begin{proof}
Global smoothness gives
\[
|\partial_iV(z[i\leftarrow x])|
\leq
|\partial_iV(0)|
+
L\bigl(|x|+\norm{z^{-i}}\bigr),
\]
and
\[
|\partial_{ii}^2V(z[i\leftarrow x])|\leq L.
\]
Since $\mu\in\cP_2(\R^{m-1})$, these estimates justify differentiation
under the integral in \eqref{eq:barV}. Thus
\[
\barV_i'(x,\mu)
=
\int_{\R^{m-1}}
\partial_iV(z[i\leftarrow x])\,\mu(\dd z^{-i})
\]
and
\[
\barV_i''(x,\mu)
=
\int_{\R^{m-1}}
\partial_{ii}^2V(z[i\leftarrow x])\,\mu(\dd z^{-i}).
\]

Fix $z^{-i}$. The vectors $z[i\leftarrow x]$ and
$z[i\leftarrow y]$ agree outside the $i$-th coordinate, so the
full-dimensional defect inequality gives
\[
\bigl(
\partial_iV(z[i\leftarrow x])
-
\partial_iV(z[i\leftarrow y])
\bigr)(x-y)
\geq
\alpha|x-y|^2-\beta.
\]
Integration with respect to $\mu$ proves
\eqref{eq:conditional-defect}. The bound \eqref{eq:conditional-smooth} follows from
$|\partial_{ii}^2V|\leq\norm{\Hess V}_{\mathrm{op}}\leq L$.

Applying the one-dimensional version of the proof of
\autoref{lem:coercivity} to $f$, using
\eqref{eq:conditional-defect}, gives the first estimate in
\eqref{eq:conditional-coercivity}. Finally,
\[
|f'(x)-f'(0)|
\leq
\int_0^{|x|}|f''(\operatorname{sgn}(x)s)|\,\dd s
\leq
L|x|,
\]
which gives the derivative bound.
\end{proof}

The conditional estimates make the coordinatewise Gibbs update well defined. For $\mu\in\cP_2(\R^{m-1})$, define
\begin{equation*}
T_i(\mu)(\dd x)
=
\frac{e^{-\barV_i(x,\mu)}}
{\displaystyle\int_{\R}e^{-\barV_i(y,\mu)}\,\dd y}\,\dd x.
\end{equation*}
The denominator is finite and strictly positive by
\eqref{eq:conditional-coercivity}. The same lower bound gives Gaussian tails,
whereas global smoothness gives at most quadratic growth of $\barV_i$.
Consequently,
\[
T_i(\mu)\in\cP_2(\R),
\]
and its entropy integral is finite.

The MFVI objective has an exact coordinatewise decomposition in terms of
$T_i$. Hence the Gibbs update is the unique minimizer of the
objective when all other marginals are held fixed.

\begin{lemma}
\label{lem:coordinate-KL}
Fix
\[
q^{-i}=\bigotimes_{k\neq i}q^k,
\]
where every $q^k\in\cP_{2,\mathrm{ac}}(\R)$ has finite entropy integral,
and let $r\in\cP_{2,\mathrm{ac}}(\R)$ also have finite entropy integral.
Set
\[
q_r
=
q^1\otimes\cdots\otimes q^{i-1}
\otimes r
\otimes q^{i+1}\otimes\cdots\otimes q^m.
\]
Then all quantities below are finite and
\begin{equation}
\label{eq:coordinate-KL}
\KL(q_r\|p)
=
\KL\bigl(r\|T_i(q^{-i})\bigr)
+
C_i(q^{-i}),
\end{equation}
where
\[
C_i(q^{-i})
=
\sum_{k\neq i}\int_{\R}q^k\log q^k
+
\log Z
-
\log\int_{\R}e^{-\barV_i(x,q^{-i})}\,\dd x
\]
is independent of $r$.
\end{lemma}

\begin{proof}
Since $q_r$ is a product measure,
\[
\begin{aligned}
\KL(q_r\|p)
&=
\int_{\R}r\log r
+
\sum_{k\neq i}\int_{\R}q^k\log q^k
+
\int_{\R}\barV_i(x,q^{-i})\,r(\dd x)
+
\log Z.
\end{aligned}
\]
If
\[
Z_i(q^{-i})
=
\int_{\R}e^{-\barV_i(x,q^{-i})}\,\dd x,
\]
then
\[
\KL\bigl(r\|T_i(q^{-i})\bigr)
=
\int_{\R}r\log r
+
\int_{\R}\barV_i(x,q^{-i})\,r(\dd x)
+
\log Z_i(q^{-i}).
\]
The entropy terms are finite by assumption. Moreover, global smoothness
implies
\[
|\barV_i(x,q^{-i})|\leq C_{q^{-i}}(1+x^2),
\]
so the potential terms are finite under the stated second-moment
assumptions. Subtracting the two identities gives
\eqref{eq:coordinate-KL}.
\end{proof}

The coordinatewise decomposition yields the Gibbs equations for every global
minimizer. In the nonconvex setting considered here,
only this necessary direction is available without additional assumptions.

\begin{proposition}
\label{prop:fixed}
Every MFVI minimizer $q^\star$ satisfies
\begin{equation}
\label{eq:fixed}
q^{\star,i}=T_i(q^{\star,-i}),
\qquad i\in[m].
\end{equation}
\end{proposition}

\begin{proof}
By \autoref{thm:existence} and its proof, every minimizer has finite second
moments and finite marginal entropy integrals. Thus
\autoref{lem:coordinate-KL} applies. Fix $i$ and hold
$q^{\star,-i}$ fixed. Minimizing the MFVI objective over the $i$-th
marginal is equivalent to minimizing
\[
\KL\bigl(r\|T_i(q^{\star,-i})\bigr).
\]
The relative entropy is nonnegative, with equality if and only if
\[
r=T_i(q^{\star,-i}).
\]
The global optimality of $q^\star$ therefore implies
\eqref{eq:fixed}.
\end{proof}

\begin{remark}
\label{rem:no-converse}
The converse of \autoref{prop:fixed} is false without additional structure.
A fixed point of
\[
T=(T_1,\ldots,T_m)
\]
is a coordinatewise minimizer of the MFVI objective, but it need not be a
global minimizer of the nonconvex product-measure problem. Additional
conditions are required to deduce global optimality or convergence of CAVI;
see \cite{AL2026,BPY2025,LZ2024}.
\end{remark}

The moment estimates for the Gibbs marginals use the following boundary
behavior in the integration-by-parts argument.

\begin{lemma}
\label{lem:boundary}
Let $f\in C^2(\R)$ satisfy \eqref{eq:conditional-coercivity}.  Then $f$ is
coercive and therefore attains its global minimum.  Let $a$ be any global
minimizer of $f$. Then, for every integer $r\geq0$,
\begin{equation}
\label{eq:boundary-limit}
\lim_{|x|\to\infty}|x-a|^r e^{-f(x)}=0.
\end{equation}
Moreover, both
\[
|x-a|^r e^{-f(x)}
\qquad\text{and}\qquad
|x-a|^r|f'(x)|e^{-f(x)}
\]
are integrable over $\R$.
\end{lemma}

\begin{proof}
The coercivity estimate in \eqref{eq:conditional-coercivity} gives
\[
|x-a|^r e^{-f(x)}
\leq
C_r(1+|x|^r)e^{-\alpha x^2/8}.
\]
The right-hand side is integrable and tends to zero as
$|x|\to\infty$, proving \eqref{eq:boundary-limit}. Since
\[
|f'(x)|\leq |f'(0)|+L|x|,
\]
the second integrability statement follows by increasing the polynomial
degree by one.
\end{proof}

Fix an MFVI minimizer $q^\star$, and define
\[
f_i(x)=\barV_i(x,q^{\star,-i}),
\qquad i\in[m].
\]
By \eqref{eq:conditional-coercivity}, each $f_i$ is coercive and therefore
attains its global minimum. Let $a_i$ be any global minimizer of $f_i$.

The Gibbs marginals satisfy uniform centered even-moment estimates.

\begin{lemma}
\label{lem:moments}
If $X_i\sim q^{\star,i}$, then
\begin{equation}
\label{eq:moment2}
\E|X_i-a_i|^2
\leq
\frac{1+\beta}{\alpha}.
\end{equation}
More generally, for every integer $k\geq2$,
\begin{equation}
\label{eq:moment-recursion}
\E|X_i-a_i|^{2k}
\leq
\frac{2k-1+\beta}{\alpha}
\E|X_i-a_i|^{2k-2}.
\end{equation}
In particular,
\begin{equation}
\label{eq:moment6}
\E|X_i-a_i|^6
\leq
\frac{(1+\beta)(3+\beta)(5+\beta)}{\alpha^3}.
\end{equation}
\end{lemma}

\begin{proof}
By \autoref{prop:fixed},
\[
q^{\star,i}(\dd x)
=
Z_i^{-1}e^{-f_i(x)}\,\dd x.
\]
Since $a_i$ is a global minimizer and $f_i\in C^2(\R)$,
\[
f_i'(a_i)=0.
\]
Applying \eqref{eq:conditional-defect} with $y=a_i$ gives
\begin{equation}
\label{eq:centered-drift}
(x-a_i)f_i'(x)
\geq
\alpha|x-a_i|^2-\beta.
\end{equation}
Integrating the derivative of
\[
(x-a_i)e^{-f_i(x)}
\]
over $[-R,R]$, and then letting $R\to\infty$, gives
\[
\E\bigl[(X_i-a_i)f_i'(X_i)\bigr]=1,
\]
where the boundary terms vanish by \autoref{lem:boundary}. Taking
expectations in \eqref{eq:centered-drift} therefore yields
\[
1
\geq
\alpha\E|X_i-a_i|^2-\beta,
\]
which proves \eqref{eq:moment2}.

For $k\geq2$, integration of the derivative of
\[
(x-a_i)^{2k-1}e^{-f_i(x)}
\]
gives
\[
\E\bigl[(X_i-a_i)^{2k-1}f_i'(X_i)\bigr]
=
(2k-1)\E|X_i-a_i|^{2k-2}.
\]
Multiplying \eqref{eq:centered-drift} by
$|x-a_i|^{2k-2}$, integrating, and using the preceding identity yields
\[
\alpha\E|X_i-a_i|^{2k}
\leq
(2k-1+\beta)\E|X_i-a_i|^{2k-2}.
\]
This gives \eqref{eq:moment-recursion}. Applying it successively for
$k=2$ and $k=3$, together with \eqref{eq:moment2}, gives
\eqref{eq:moment6}.
\end{proof}

The centered marginal bounds imply a dimension-explicit estimate for the
stationary gradient moment, which controls both the stochastic batch error and
the local discretization error in the particle analysis.

\begin{lemma}
\label{lem:gradient-moment}
If $X\sim q^\star$, then
\begin{equation}
\label{eq:grad-zero}
\E\grad V(X)=0,
\end{equation}
and
\begin{equation}
\label{eq:Gamma}
\Gamma_\star
:=
\E\norm{\grad V(X)}^2
\leq
\frac{mL^2(1+\beta)}{\alpha}.
\end{equation}
\end{lemma}

\begin{proof}
Since $q^\star$ is a product measure, Fubini's theorem gives
\[
\E\partial_iV(X)
=
\E f_i'(X_i).
\]
By \autoref{lem:boundary},
\[
\int_{\R}f_i'(x)e^{-f_i(x)}\,\dd x
=
-\int_{\R}\frac{\dd}{\dd x}e^{-f_i(x)}\,\dd x
=
0.
\]
Hence
\[
\E\partial_iV(X)=0
\]
for every $i$, which proves \eqref{eq:grad-zero}.

Let $X'$ be an independent copy of $X$. Since
$\E\grad V(X)=0$,
\[
\begin{aligned}
2\E\norm{\grad V(X)}^2
&=
\E\norm{\grad V(X)-\grad V(X')}^2\\
&\leq
L^2\E\norm{X-X'}^2.
\end{aligned}
\]
Furthermore,
\[
\E\norm{X-X'}^2
=
2\sum_{i=1}^m\Var(X_i).
\]
For every $i$,
\[
\Var(X_i)
=
\inf_{a\in\R}\E|X_i-a|^2
\leq
\E|X_i-a_i|^2
\leq
\frac{1+\beta}{\alpha}
\]
by \autoref{lem:moments}. Combining these estimates proves
\eqref{eq:Gamma}.
\end{proof}

The sixth-moment estimate also controls approximation of $q^\star$ by a
product of one-dimensional empirical measures. We use the following
one-dimensional bound.

\begin{lemma}
\label{lem:FG}
There exists a universal constant $C_{\mathrm{FG}}<\infty$ such that, for
every $\mu\in\cP_6(\R)$, every $a\in\R$, and
\[
\widehat\mu_N
=
\frac1N\sum_{j=1}^N\delta_{Z_j},
\qquad
Z_j\stackrel{\mathrm{i.i.d.}}{\sim}\mu,
\]
one has
\begin{equation*}
\E\cW^2(\widehat\mu_N,\mu)
\leq
\frac{C_{\mathrm{FG}}}{\sqrt N}
\left(
\int_{\R}|x-a|^6\,\mu(\dd x)
\right)^{1/3}.
\end{equation*}
\end{lemma}

\begin{proof}
Apply \cite[Theorem~1]{FG2015} with dimension $d=1$, transport exponent $2$,
and moment exponent $6$. The result there is stated with
moments centered at the origin, and translation invariance of the
Wasserstein distance permits centering at an arbitrary $a\in\R$.
\end{proof}

Let $Y\in\R^{m\times N}$ have independent columns with common law
$q^\star$, and define
\[
q_Y
=
\bigotimes_{i=1}^m
\left(
\frac1N\sum_{j=1}^N\delta_{Y^{i,j}}
\right),
\qquad
\varepsilon_N^2(q^\star)
=
\E\cW^2(q_Y,q^\star).
\]

Product additivity gives the empirical approximation bound used below.

\begin{corollary}
\label{cor:particle}
Under the preceding assumptions,
\begin{equation}
\label{eq:epsN2}
\varepsilon_N^2(q^\star)
\leq
\frac{C_{\mathrm{FG}}m}{\alpha\sqrt N}
\bigl((1+\beta)(3+\beta)(5+\beta)\bigr)^{1/3}.
\end{equation}
Consequently,
\begin{equation}
\label{eq:epsN}
\varepsilon_N(q^\star)
\leq
C
\frac{\sqrt m}{\sqrt\alpha}
\bigl((1+\beta)(3+\beta)(5+\beta)\bigr)^{1/6}
N^{-1/4},
\end{equation}
where $C=\sqrt{C_{\mathrm{FG}}}$ is universal.
\end{corollary}

\begin{proof}
Additivity of the squared Wasserstein distance for product measures gives
\[
\cW^2(q_Y,q^\star)
=
\sum_{i=1}^m
\cW^2\left(
\frac1N\sum_{j=1}^N\delta_{Y^{i,j}},
q^{\star,i}
\right).
\]
For each $i$, apply \autoref{lem:FG} with
\[
\mu=q^{\star,i},
\qquad
a=a_i,
\]
and then use \eqref{eq:moment6}. Summing over $i$ yields
\eqref{eq:epsN2}. Taking square roots gives \eqref{eq:epsN}.
\end{proof}

The empirical rate improves under stronger one-dimensional regularity.

\begin{remark}
\label{rem:empirical-rates}
The rate $N^{-1/4}$ in \eqref{eq:epsN} concerns the root-mean-square
$\cW$ error and uses only a sixth-moment bound. It is therefore a worst-case estimate rather than a sharp prediction for
smooth Gibbs marginals.

When $\beta=0$, \eqref{eq:conditional-defect} implies that every
fixed-point potential $f_i$ is $\alpha$-strongly convex. Hence each
marginal $q^{\star,i}$ is $\alpha$-strongly log-concave. For $N\geq2$,
the one-dimensional empirical estimate of Bobkov and Ledoux \cite{BL2019},
used in \cite[Lemma~6.5]{DWZZ2026}, gives
\[
\E\cW^2(\widehat q_N^i,q^{\star,i})
\leq
C\frac{\log N}{\alpha N}.
\]
By product additivity,
\[
\varepsilon_N(q^\star)
\leq
C\sqrt{\frac{m\log N}{\alpha N}}.
\]

Full log-concavity is not necessary for a sharper one-dimensional rate.
A precise criterion is available from the one-dimensional empirical theory
in \cite[Chapter~5]{BL2019}.  If $\mu$ is supported on an interval, has an
a.e. positive density $f$ there, and the inverse distribution function has
the regularity specified in that result, then the standard root-mean-square
$N^{-1/2}$ rate is characterized by finiteness of
\begin{equation}
\label{eq:quantile-regularity}
J_2(\mu)
=
\int_0^1
\frac{t(1-t)}{f(F^{-1}(t))^2}\,\dd t
=
\int_\R\frac{F(x)(1-F(x))}{f(x)}\,\dd x.
\end{equation}
In particular, $J_2(\mu)<\infty$ yields
$\E\cW^2(\widehat\mu_N,\mu)=O(N^{-1})$.  Hence, if the fixed-point marginals
satisfy bounds
\[
\E\cW^2(\widehat q_N^i,q^{\star,i})\leq\frac{C_i}{N},
\]
then product additivity gives
\begin{equation}
\label{eq:epsN-standard-rate}
\varepsilon_N(q^\star)
\leq
\left(\frac1N\sum_{i=1}^m C_i\right)^{1/2}.
\end{equation}
The criterion applies without log-concavity, provided the one-dimensional
marginals have sufficiently regular quantiles.

Sub-Gaussian tails alone do not imply this rate. Even the standard Gaussian
has $J_2=\infty$ and fails to satisfy a pure $O(N^{-1})$ mean-squared
quadratic empirical-transport bound.
More generally, a tail upper bound does not control the reciprocal density in
\eqref{eq:quantile-regularity}, especially across low-density bottlenecks.
Quadratic coercivity alone therefore does not justify replacing
\eqref{eq:epsN} by an $N^{-1/2}$ root-mean-square rate uniformly over the
present class.  We retain the Fournier--Guillin estimate in the main theorem
because it follows from assumptions already used elsewhere;
\eqref{eq:epsN-standard-rate} gives the sharper rate when quantile regularity
is available.  The smooth benchmark introduced in
\autoref{sec:benchmark} exhibits an empirical rate close to $N^{-1/2}$; see
panel~\textup{(a)} of \autoref{fig:particle-step}.
\end{remark}

These results identify the stationary Gibbs law and provide the gradient and
empirical-moment bounds needed for the continuous and particle couplings.

\section{Independent-projection dynamics}\label{sec:independent-projection}

For a product measure $\mu=\otimes_{i=1}^m\mu^i$, define
\[
b_i(x,\mu)=-\barV_i'(x,\mu^{-i}).
\]
For an $\R^m$-valued process $Y$, write
\[
\Law^{\otimes,-i}(Y_t)
:=\bigotimes_{k\ne i}\Law(Y_{t,k}).
\]
The independent-projection McKean--Vlasov equation is
\begin{equation}\label{eq:MV}
\dd Y_{t,i}
=-\barV_i'\bigl(Y_{t,i},\Law^{\otimes,-i}(Y_t)\bigr)\dd t
+\sqrt2\dd B_{t,i},
\qquad i\in[m].
\end{equation}
The product of the marginal laws is part of the definition: the coefficient does not use a possibly correlated joint law of $Y_t^{-i}$.

\begin{lemma}\label{lem:coefficient-lip}
Under \autoref{ass:smooth}, for $x,y\in\R$ and $\mu,\nu\in\cP_2(\R^{m-1})$,
\begin{equation*}
|\barV_i'(x,\mu)-\barV_i'(y,\nu)|
\le L\bigl(|x-y|+\cW(\mu,\nu)\bigr).
\end{equation*}
\end{lemma}

\begin{proof}
Let $\gamma\in\Pi(\mu,\nu)$.  For $(z,w)\in\R^{m-1}\times\R^{m-1}$, global smoothness gives
\[
|\partial_iV(z[i\leftarrow x])-\partial_iV(w[i\leftarrow y])|
\le L\sqrt{|x-y|^2+\|z-w\|^2}
\le L(|x-y|+\|z-w\|).
\]
Integrate with respect to $\gamma$, use Cauchy--Schwarz, and minimize over couplings.
\end{proof}

We begin with global well-posedness and preservation of the product structure.

\begin{theorem}\label{thm:wellposed}
Under \autoref{ass:smooth}, equation~\eqref{eq:MV} has a unique strong solution for every product initial law $\mu_0\in\cP_2(\R^m)$.  For every $t\ge0$, $\Law(Y_t)$ is a product measure.
\end{theorem}

\begin{proof}
A fixed-point argument also records preservation of independence; see
\cite{CD2018} for standard McKean--Vlasov well-posedness theory.  Write the prescribed product initial law as
\[
\mu_0=\bigotimes_{i=1}^m\mu_0^i.
\]
Fix $T>0$, and let
\[
\mathcal C_T(\mu_0)
=
\left\{
\boldsymbol\mu=(\mu^1,\ldots,\mu^m)
\in C\bigl([0,T];\cP_2(\R)^m\bigr):
\mu^i(0)=\mu_0^i,\ i\in[m]
\right\}.
\]
For $0\le t\le T$, set
\[
d_t^2(\boldsymbol\mu,\boldsymbol\nu)
=
\sup_{0\le r\le t}\sum_{i=1}^m\cW^2(\mu_r^i,\nu_r^i).
\]
Then $d_T$ is a metric on $\mathcal C_T(\mu_0)$, and this space is complete because $\cP_2(\R)$ is complete in $\cW$ and the prescribed-initial-value condition is closed.  For a deterministic curve $\boldsymbol\mu\in\mathcal C_T(\mu_0)$, solve the decoupled SDEs
\begin{equation}\label{eq:frozen-law-sde}
\dd X_{t,i}^{\boldsymbol\mu}
=-\barV_i'\left(X_{t,i}^{\boldsymbol\mu},\bigotimes_{k\ne i}\mu_t^k\right)\dd t
+\sqrt2\dd B_{t,i},
\qquad X_{0,i}^{\boldsymbol\mu}\sim\mu_0^i,
\end{equation}
using the prescribed independent initial variables and independent Brownian motions.  \autoref{lem:coefficient-lip} gives global Lipschitz continuity and linear growth, so these equations have unique strong solutions.  Define
\[
\Phi(\boldsymbol\mu)_t^i=\Law(X_{t,i}^{\boldsymbol\mu}).
\]
Couple $X^\mu$ and $X^\nu$ synchronously, using the same initial
variables and the same Brownian motions.  Set
\[
A_t
:=
\sup_{0\le r\le t}
\sum_{i=1}^m
\E\left|X_{r,i}^{\mu}-X_{r,i}^{\nu}\right|^2.
\]
By \autoref{lem:coefficient-lip}, Cauchy--Schwarz in time, and
additivity of $\cW^2$ for product measures,
\[
\begin{aligned}
\sum_{i=1}^m
\E\left|X_{t,i}^{\mu}-X_{t,i}^{\nu}\right|^2
&\le
2L^2t\int_0^t
\sum_{i=1}^m
\left[
\E\left|X_{s,i}^{\mu}-X_{s,i}^{\nu}\right|^2
+
\cW^2(\mu_s^{-i},\nu_s^{-i})
\right]\dd s \\
&\le
C_{m,L}t\int_0^t
\left[
A_s+d_s^2(\boldsymbol\mu,\boldsymbol\nu)
\right]\dd s.
\end{aligned}
\]
Consequently, for $t\le T$,
\[
A_t
\le
C_{m,L}T\int_0^t A_s\,\dd s
+
C_{m,L}Tt\,d_T^2(\boldsymbol\mu,\boldsymbol\nu).
\]
Gronwall's inequality gives
\[
A_T
\le
C_{m,L}T^2e^{C_{m,L}T^2}
d_T^2(\boldsymbol\mu,\boldsymbol\nu).
\]
Since the synchronous coupling is admissible,
\[
d_T^2\bigl(\Phi(\boldsymbol\mu),\Phi(\boldsymbol\nu)\bigr)
\le A_T.
\]
Therefore
\[
d_T^2\bigl(\Phi(\boldsymbol\mu),\Phi(\boldsymbol\nu)\bigr)
\le
C_{m,L}T^2e^{C_{m,L}T^2}
d_T^2(\boldsymbol\mu,\boldsymbol\nu),
\]
and $\Phi$ is a contraction for $T>0$ sufficiently small.
Iterating the construction on consecutive intervals of this length gives a unique global fixed point $\boldsymbol\mu$.

At the fixed point, the coordinates in \eqref{eq:frozen-law-sde} are functions of independent initial variables and independent Brownian motions with deterministic coefficient curves.  They are therefore independent, and
\[
\Law(X_t)=\bigotimes_{i=1}^m\mu_t^i.
\]
The resulting product process solves \eqref{eq:MV}. Conversely, the
marginal-law vector of any solution is a fixed point of $\Phi$, because the
coefficients in \eqref{eq:MV} depend only on the products of the marginal
laws. Uniqueness of the fixed point determines these marginal curves. With
the curves fixed, pathwise uniqueness for the decoupled equations gives
strong uniqueness. Product initial data and independent Brownian drivers then
preserve independence of the coordinates.
\end{proof}

A synchronous coupling gives the continuous-time stability estimate.

\begin{theorem}\label{thm:continuous}
Let $\mu_t$ and $\nu_t$ solve \eqref{eq:MV} with product initial laws.  Start them from the product of optimal marginal couplings and couple corresponding coordinates synchronously.  Denote the coupled processes by $X_t,Y_t$, and set
\begin{equation*}
\Delta(t)=\E\mathfrak d_\alpha(X_t,Y_t).
\end{equation*}
Then
\begin{equation}\label{eq:continuous-bound}
\cW^2(\mu_t,\nu_t)
\le
e^{-2\alpha t}\cW^2(\mu_0,\nu_0)
+2\int_0^te^{-2\alpha(t-s)}\Delta(s)\dd s.
\end{equation}
If $\mathfrak d_\alpha\le\beta$, then
\begin{equation}\label{eq:continuous-beta}
\cW^2(\mu_t,\nu_t)
\le
e^{-2\alpha t}\cW^2(\mu_0,\nu_0)
+\frac\beta\alpha(1-e^{-2\alpha t}).
\end{equation}
The same conclusions hold with $\nu_t=q^\star$ for any product law $q^\star\in\cP_2(\R^m)$ satisfying the coordinatewise Gibbs identities \eqref{eq:fixed}; in particular, they hold for every MFVI minimizer under the uniform-defect assumptions of \autoref{thm:existence}.
\end{theorem}

\begin{proof}
By construction, the initial coordinate pairs are independent across $i$, and the synchronous Brownian drivers used for different coordinates are independent.  Since the marginal-law curves are deterministic by \autoref{thm:wellposed}, the coordinate pairs $(X_{t,i},Y_{t,i})$ therefore remain independent across $i$.  Let
\[
\mathcal F_{t,i}=\sigma(X_{t,i},Y_{t,i}).
\]
The independence of the remaining coordinate pairs implies
\begin{equation*}
\barV_i'(X_{t,i},\mu_t^{-i})-
\barV_i'(Y_{t,i},\nu_t^{-i})
=
\E[\partial_iV(X_t)-\partial_iV(Y_t)\mid\mathcal F_{t,i}].
\end{equation*}
The independence of the coordinate pairs is used in this conditional
expectation identity.  Conditional on $(X_{t,i},Y_{t,i})$, the random vectors $X_t^{-i}$ and
$Y_t^{-i}$ retain marginal laws $\mu_t^{-i}$ and $\nu_t^{-i}$,
respectively.  No independence between $X_t^{-i}$ and $Y_t^{-i}$ is required;
their joint conditional law is generally the product of the chosen
coordinate-pair couplings, not $\mu_t^{-i}\otimes\nu_t^{-i}$.
Since $X_{t,i}-Y_{t,i}$ is $\mathcal F_{t,i}$-measurable,
\[
\begin{aligned}
&\sum_{i=1}^m\E\left[(X_{t,i}-Y_{t,i})
\bigl(\barV_i'(X_{t,i},\mu_t^{-i})-
\barV_i'(Y_{t,i},\nu_t^{-i})\bigr)\right]\\
&\qquad=
\E\langle X_t-Y_t,\nabla V(X_t)-\nabla V(Y_t)\rangle\\
&\qquad\ge
\alpha\E\|X_t-Y_t\|^2-\Delta(t).
\end{aligned}
\]
The Brownian terms cancel under synchronous coupling, and It\^o's formula gives
\[
\frac{\dd}{\dd t}\E\|X_t-Y_t\|^2
\le-2\alpha\E\|X_t-Y_t\|^2+2\Delta(t).
\]
Variation of constants gives \eqref{eq:continuous-bound}. The initial product
coupling is optimal by additivity of squared Wasserstein distance, whereas the
later-time coupling is merely admissible; its mean-square cost therefore
bounds $\cW^2(\mu_t,\nu_t)$ from above. Under the uniform defect bound,
\eqref{eq:continuous-beta} follows.

If $q^\star$ satisfies the coordinatewise Gibbs identities, then
\[
q^{\star,i}(\dd x)\propto
\exp\{-\barV_i(x,q^{\star,-i})\}\,\dd x.
\]
When the joint law is $q^\star$, the drift of coordinate $i$ is
$-\barV_i'(\cdot,q^{\star,-i})$, the logarithmic derivative of the invariant density $q^{\star,i}$.  Each scalar Langevin equation therefore preserves its marginal $q^{\star,i}$.  Because the coordinates are driven by independent Brownian motions and the initial law is the product $q^\star$, their independence is preserved, and the joint product measure $q^\star$ is stationary for \eqref{eq:MV}.
\end{proof}

The same estimate controls the separation between distinct MFVI minimizers.

\begin{corollary}
\label{cor:minimizer-diameter}
Assume the uniform defect bound $\mathfrak d_\alpha\leq\beta$.  If
$q^\star$ and $\widetilde q^\star$ are two MFVI minimizers, then
\begin{equation}
\label{eq:minimizer-diameter}
\cW(q^\star,\widetilde q^\star)
\leq
\sqrt{\frac{\beta}{\alpha}}.
\end{equation}
In particular, $\beta=0$ implies uniqueness, while for $\beta>0$ the uniform
curvature defect controls the diameter of the entire minimizer set even
though it does not select one minimizer dynamically.
\end{corollary}

\begin{proof}
By \autoref{prop:fixed}, both minimizers satisfy the coordinatewise Gibbs
identities and hence are stationary for \eqref{eq:MV}.  Apply
\eqref{eq:continuous-beta} with
$\mu_t\equiv q^\star$ and
$\nu_t\equiv\widetilde q^\star$.  Since the left-hand side is independent of
$t$,
\[
\cW^2(q^\star,\widetilde q^\star)
\leq
e^{-2\alpha t}\cW^2(q^\star,\widetilde q^\star)
+\frac{\beta}{\alpha}(1-e^{-2\alpha t}).
\]
For any $t>0$, cancellation of the factor $1-e^{-2\alpha t}$ gives
\eqref{eq:minimizer-diameter}.
\end{proof}

Moving-window averages also control the convolution term in the continuous-time
estimate.

\begin{remark}
\label{rem:time-averaged-defect}
The exponential convolution in \eqref{eq:continuous-bound} is well suited to
mixing or ergodic estimates, but an ordinary long-time Ces\`aro average alone
does not control it because defects occurring near the terminal time receive
the largest weight.  A sufficient condition is a uniform moving-window bound: if, for some $\tau>0$ and $\overline\Delta<\infty$,
\[
\sup_{r\geq0}\int_r^{r+\tau}\Delta(s)\,\dd s
\leq
\tau\overline\Delta,
\]
then partitioning the past into intervals of length $\tau$ gives
\begin{equation}
\label{eq:moving-average-defect}
\int_0^t e^{-2\alpha(t-s)}\Delta(s)\,\dd s
\leq
\frac{\tau\overline\Delta}{1-e^{-2\alpha\tau}},
\qquad t\geq0.
\end{equation}
Any ergodic or mixing estimate that controls the defect uniformly on sliding
windows can therefore be combined with \autoref{thm:continuous}.
If $\Delta(s)\leq\overline\Delta$ pointwise, letting $\tau\downarrow0$ in
\eqref{eq:moving-average-defect} recovers the familiar residual
$\overline\Delta/(2\alpha)$ inside the convolution.
\end{remark}

\section{Finite-batch PAVI}\label{sec:finite-batch-pavi}

For $A\in\R^{m\times N}$, define the row empirical measures and their product by
\[
q_A^i=\frac1N\sum_{j=1}^N\delta_{A^{i,j}},
\qquad
q_A=\bigotimes_{i=1}^mq_A^i.
\]
A sample from $q_A$ is formed by drawing the row indices independently.  Drawing a common column index would sample the column empirical law, which is a different measure.

\begin{algorithm}[t]
\caption{Finite-batch PAVI}\label{alg:pavi}
\begin{algorithmic}[1]
\Require Potential $V$, initial array $X_0\in\R^{m\times N}$, step size $h>0$, batch size $B\ge1$, iterations $T$.
\For{$n=0,\ldots,T-1$}
  \State Draw $z_n^{:,1},\ldots,z_n^{:,B}$ independently from the product empirical law $q_{X_n}$.
  \For{$i=1,\ldots,m$}
    \State Define
    \[
    g_n^i(x)=\frac1B\sum_{b=1}^B
    \partial_iV(z_n^{1,b},\ldots,z_n^{i-1,b},x,z_n^{i+1,b},\ldots,z_n^{m,b}).
    \]
    \State Draw $\xi_n^{i,:}\sim\mathcal N(0,I_N)$, independently.
    \For{$j=1,\ldots,N$}
      \State $X_{n+1}^{i,j}\gets X_n^{i,j}-hg_n^i(X_n^{i,j})+\sqrt{2h}\,\xi_n^{i,j}$.
    \EndFor
  \EndFor
\EndFor
\State \Return $X_T$.
\end{algorithmic}
\end{algorithm}

A direct implementation has the following generic work and storage costs.

\begin{proposition}\label{prop:complexity}
Suppose one evaluation of a coordinate derivative $\partial_iV$ has cost $C_{\partial V}$.  A direct implementation of Algorithm~\ref{alg:pavi} uses
\[
O(mNB\,C_{\partial V})
\]
work and $O(m(N+B))$ storage per iteration, excluding the storage used internally by the derivative routine.  By contrast, direct evaluation of every full projected drift over the product empirical law uses $mN^m$ coordinate-derivative evaluations in the generic unstructured case.
\end{proposition}

\begin{proof}
For each of the $mN$ particle coordinates, the batch estimator averages $B$ coordinate derivatives.  The particle array and product batch require $mN$ and $mB$ numbers, respectively.  For the full projected drift, fixing $(i,j)$ leaves $m-1$ empirical coordinates to average, producing $N^{m-1}$ terms; multiplication by $mN$ gives $mN^m$.
\end{proof}

Model structure can reduce this generic cost.

\begin{remark}
The count in \autoref{prop:complexity} is a generic upper bound.  Additive or low-rank interactions may permit the batch averages to be reduced to a small number of sufficient statistics.  The benchmark in Section~8 has this structure, and its specialized implementation costs $O(mN+mB+\operatorname{nnz}(J))$ per iteration.  The companion code repository \cite{NV2026} includes both the generic implementation and the specialized benchmark code.
\end{remark}

\section{Stationary-array coupling and the one-step estimate}\label{sec:stationary-array}

\paragraph{Standing initial-moment assumption.}
Throughout Sections~6--7, the initial array is assumed to satisfy
\begin{equation}\label{eq:standing-initial-moment}
\E\left[\frac1N\sum_{j=1}^N\|X_0^{:,j}\|^2\right]<\infty.
\end{equation}
This condition ensures finiteness of the coupling costs and conditional second
moments used below.  Under the uniform-defect assumptions, \autoref{lem:coercivity} and the moment estimates of Section~3 supply the corresponding moments for the stationary comparison array.

Fix a product probability law $q^\star\in\cP_6(\R^m)$ satisfying the stationary coordinatewise Gibbs identities
\begin{equation}\label{eq:stationary-fixed-point}
q^{\star,i}(\dd x)
=
\frac{e^{-\barV_i(x,q^{\star,-i})}}
{\int_\R e^{-\barV_i(y,q^{\star,-i})}\dd y}\,\dd x,
\qquad i\in[m].
\end{equation}
Set
\[
\Gamma_\star=\E_{q^\star}\|\nabla V(X)\|^2<\infty,
\]
which is finite because global smoothness gives at most linear growth of $\nabla V$.  Under the uniform-defect assumptions of \autoref{sec:variational-foundations}, every MFVI minimizer satisfies \eqref{eq:stationary-fixed-point} by \autoref{prop:fixed}, and Lemmas~\ref{lem:moments}--\ref{lem:gradient-moment} provide the required moments and the explicit bound \eqref{eq:Gamma}.

All random variables below are defined on a common filtered probability space satisfying the usual conditions.  The batch samples used by PAVI are independent of the Brownian motions and of the initial stationary array.

\subsection{Construction of the stationary array}

Let $B^{i,j}$, $(i,j)\in[m]\times[N]$, be independent standard Brownian motions.  Run Algorithm~\ref{alg:pavi} with
\[
\xi_n^{i,j}=
\frac{B_{(n+1)h}^{i,j}-B_{nh}^{i,j}}{\sqrt h}.
\]
Let $Y_0$ have independent columns with law $q^\star$, independently of $X_0$ and $B$.

At time $nh$, choose for every row $i$ an optimal matching permutation $\tau_{n,i}$ such that
\begin{equation}\label{eq:rowmatch}
\cW^2(q_{X_n}^i,q_{Y_{nh}}^i)
=
\frac1N\sum_{j=1}^N
|X_n^{i,j}-Y_{nh}^{i,\tau_{n,i}(j)}|^2.
\end{equation}
In one dimension this permutation is obtained by sorting both rows.  We resolve ties using the original particle labels, which makes $\tau_{n,i}$ a measurable function of the two rows.

On $[nh,(n+1)h]$, define
\begin{equation}\label{eq:Ycoupled}
\dd Y_t^{i,j}
=-\barV_i'(Y_t^{i,j},q^{\star,-i})\dd t
+\sqrt2\dd B_t^{i,\tau_{n,i}^{-1}(j)}.
\end{equation}
Hence the matched particle $Y^{i,\tau_{n,i}(j)}$ uses the same Brownian
increment as $X^{i,j}$ during the current time step.

The predictable relabeling preserves stationarity of the comparison array.

\begin{lemma}\label{lem:stationary-array}
For every $t\ge0$, the $N$ columns of $Y_t$ are independent with common law $q^\star$.  Consequently,
\begin{equation}\label{eq:stationary-empirical-error}
\E \cW^2(q_{Y_t},q^\star)=\varepsilon_N^2(q^\star).
\end{equation}
\end{lemma}

\begin{proof}
Stack the $mN$ Brownian motions into a vector $B_t$.  On $(nh,(n+1)h]$, let $Q_n$ be the block-diagonal permutation matrix whose $i$-th block represents $\tau_{n,i}^{-1}$.  The matrix $Q_n$ is measurable with respect to the sigma-field at time $nh$, and it is orthogonal.  Define
\[
\overline B_t=\int_0^tQ_s\dd B_s,
\qquad Q_s=Q_n\quad\text{for }s\in(nh,(n+1)h].
\]
It is a continuous local martingale with
\[
\langle\overline B^a,\overline B^b\rangle_t
=\delta_{ab}t.
\]
L\'evy's characterization therefore shows that $\overline B$ is an
$mN$-dimensional standard Brownian motion with respect to the underlying
filtration. Equation~\eqref{eq:Ycoupled} is the corresponding collection of
scalar SDEs driven by the components of $\overline B$.

For fixed $i$, the invariant density of the scalar SDE is proportional to
\[
e^{-\barV_i(x,q^{\star,-i})}.
\]
By the assumed fixed-point identity \eqref{eq:stationary-fixed-point},
the invariant density of the $i$-th scalar equation is
$q^{\star,i}$. The $mN$ equations are decoupled, their driving Brownian components are independent, and the entries of $Y_0$ are independent with laws $q^{\star,i}$.  Hence the entries remain independent and stationary.  Their columns therefore have common product law $q^\star$, and \eqref{eq:stationary-empirical-error} follows from the definition of $\varepsilon_N$.
\end{proof}

Set
\begin{equation*}
\mathcal E_n=
\E \cW^2(q_{X_n},q_{Y_{nh}}).
\end{equation*}
By product additivity (\autoref{lem:product-additivity}) and the optimal rowwise matchings \eqref{eq:rowmatch},
\begin{equation*}
\mathcal E_n
=
\frac1N\sum_{i=1}^m\sum_{j=1}^N
\E|X_n^{i,j}-Y_{nh}^{i,\tau_{n,i}(j)}|^2.
\end{equation*}
This identity combines additivity of squared Wasserstein distance for product
measures with the one-dimensional optimal assignments in \eqref{eq:rowmatch}.

\subsection{Measure sensitivity of the projected drift}

For a product measure $\mu$, define
\[
B_\mu(z)
=
(\barV_1'(z_1,\mu^{-1}),\ldots,
\barV_m'(z_m,\mu^{-m})).
\]
For $z,u\in\R^m$, introduce the matrix
\begin{equation}\label{eq:hybrid-matrix}
[A_z(u)]_{ik}
=
\begin{cases}
\partial_{ik}^2V(u_1,\ldots,u_{i-1},z_i,u_{i+1},\ldots,u_m),&k\ne i,\\
0,&k=i,
\end{cases}
\end{equation}
and set
\begin{equation}\label{eq:kappa-definition}
\kappa=\sup_{z,u\in\R^m}\|A_z(u)\|_{\mathrm{op}}.
\end{equation}
Each row of $A_z(u)$ has Euclidean norm at most $L$, so
\begin{equation}\label{eq:kappa-bound}
\kappa\le\sup_{z,u}\|A_z(u)\|_{\mathrm F}\le\sqrt mL.
\end{equation}
We retain a separate constant $\kappa$ because interaction structure may make
it much smaller than $\sqrt mL$.

A Schur bound gives dimension-independent control of $\kappa$ under sparse
interaction structure.

\begin{remark}
\label{rem:kappa-structure}
The factor $\sqrt m$ in \eqref{eq:kappa-bound} comes from the universal
Frobenius estimate and need not be sharp.  Let
\[
K_\infty=\sup_{z,u}\|A_z(u)\|_\infty,
\qquad
K_1=\sup_{z,u}\|A_z(u)\|_1,
\]
where the two matrix norms are the maximal absolute row and column sums.
The Schur bound gives
\begin{equation*}
\kappa
\leq
\sqrt{K_1K_\infty}.
\end{equation*}
Consequently, if every coordinate interacts with at most $d$ other
coordinates and the relevant cross derivatives satisfy
$|\partial_{ik}^2V|\leq\ell$, with the same bound for the number of
interactions entering each coordinate, then
$K_1,K_\infty\leq d\ell$ and
\[
\kappa\leq d\ell,
\]
independently of $m$.  Low-rank and related structures are also covered by
\eqref{eq:kappa-definition} when their operator norms are uniformly controlled;
the benchmark in
\autoref{sec:benchmark} is an example.
\end{remark}

With $\kappa$ fixed, the projected drift is Lipschitz with respect to the product law.

\begin{lemma}\label{lem:measure-lip}
For all product measures $\mu,\nu\in\cP_2(\R^m)$ and all $z\in\R^m$,
\begin{equation*}
\|B_\mu(z)-B_\nu(z)\|
\le\kappa \cW(\mu,\nu).
\end{equation*}
\end{lemma}

\begin{proof}
For fixed $z$, define $F_z:\R^m\to\R^m$ by
\[
[F_z(u)]_i
=
\partial_iV(u_1,\ldots,u_{i-1},z_i,u_{i+1},\ldots,u_m).
\]
Its Jacobian is $A_z(u)$, so the fundamental theorem of calculus and \eqref{eq:kappa-definition} give
\[
\|F_z(u)-F_z(v)\|\le\kappa\|u-v\|.
\]
If $\gamma\in\Pi(\mu,\nu)$, then
\[
B_\mu(z)-B_\nu(z)
=
\int(F_z(u)-F_z(v))\,\gamma(\dd u,\dd v).
\]
Jensen's inequality gives
\[
\|B_\mu(z)-B_\nu(z)\|^2
\le\kappa^2\int\|u-v\|^2\gamma(\dd u,\dd v).
\]
Minimizing over $\gamma$ proves the result.  The rows of $A_z(u)$ are evaluated
at different hybrid points, so the universal estimate \eqref{eq:kappa-bound}
is obtained rowwise through the Frobenius norm.
\end{proof}

\subsection{Stochastic-gradient error}

Define the full empirical-drift update
\begin{equation*}
\widetilde X_{n+1}^{i,j}
=
X_n^{i,j}
-h\barV_i'(X_n^{i,j},q_{X_n}^{-i})
+\sqrt2(B_{(n+1)h}^{i,j}-B_{nh}^{i,j}).
\end{equation*}
Let
\[
E_n^{i,j}=X_{n+1}^{i,j}-\widetilde X_{n+1}^{i,j}
=-h\bigl(g_n^i(X_n^{i,j})-
\barV_i'(X_n^{i,j},q_{X_n}^{-i})\bigr).
\]

The batch estimator is conditionally unbiased and satisfies the following
variance bound.

\begin{lemma}\label{lem:batch}
For every $n$,
\begin{equation}\label{eq:batch}
\frac1N\sum_{i,j}\E|E_n^{i,j}|^2
\le
\frac{2h^2}{B}
\left(L^2\mathcal E_n+\Gamma_\star\right).
\end{equation}
If $H$ is square-integrable and measurable with respect to $X_n,Y_{nh}$ and the Brownian increments on $[nh,(n+1)h]$, then
\begin{equation}\label{eq:batch-orthogonality}
\E[E_n^{i,j}H]=0.
\end{equation}
\end{lemma}

\begin{proof}
Conditional on $X_n$, the $B$ summands in $g_n^i(X_n^{i,j})$ are independent and have mean $\barV_i'(X_n^{i,j},q_{X_n}^{-i})$.  Hence
\[
\begin{aligned}
\E[|E_n^{i,j}|^2\mid X_n]
&\le\frac{h^2}{B}
\int
|\partial_iV(X_n^{i,j},z^{-i})|^2q_{X_n}^{-i}(\dd z^{-i}).
\end{aligned}
\]
Averaging over $j$ and summing over $i$ gives
\begin{equation}\label{eq:batch-secondmoment}
\frac1N\sum_{i,j}\E|E_n^{i,j}|^2
\le\frac{h^2}{B}
\E\int\|\nabla V(z)\|^2q_{X_n}(\dd z).
\end{equation}

Condition on the two arrays and couple $Z\sim q_{X_n}$ and $U\sim q_{Y_{nh}}$ optimally by taking the product of the rowwise matchings.  Then
\[
\E\|Z-U\|^2=\mathcal E_n.
\]
Conditional on $Y_{nh}$, the law of $U$ is $q_{Y_{nh}}$. Averaging over the
stationary array, each independently selected row entry has law
$q^{\star,i}$, and the row selections are independent. Hence the
unconditional law of $U$ is $q^\star$. Global smoothness gives
\[
\E\|\nabla V(Z)\|^2
\le2L^2\mathcal E_n+2\E_{q^\star}\|\nabla V(U)\|^2
=2L^2\mathcal E_n+2\Gamma_\star.
\]
Substitution in \eqref{eq:batch-secondmoment} proves \eqref{eq:batch}.

Finally, the fresh batch randomness at step $n$ is independent of
$\sigma(X_n,Y_{nh})$ and of the Brownian increments on
$[nh,(n+1)h]$.  Since the batch error is centered conditional on $X_n$,
conditioning first on $X_n,Y_{nh}$ and the Brownian increments gives
\eqref{eq:batch-orthogonality}.
\end{proof}

\subsection{Empirical Euler comparison and encountered defect}

The rowwise optimal matchings are individually optimal but do not by
themselves produce independent full-dimensional coordinate pairs.  To apply
the pointwise full-dimensional defect while preserving each rowwise matching
cost, we independently randomize the particle labels in the different rows.
Let $\sigma_1,\ldots,\sigma_m$ be independent uniform permutations of $[N]$, independent of all preceding random variables.  Define
\begin{equation*}
U^{i,j}=X_n^{i,\sigma_i(j)},
\qquad
Z^{i,j}=Y_{nh}^{i,\tau_{n,i}(\sigma_i(j))}.
\end{equation*}
Conditional on the arrays, the coordinate pairs $(U^{i,j},Z^{i,j})$ are independent over $i$, and each is the optimal empirical coupling specified by \eqref{eq:rowmatch}.  Define
\begin{equation}\label{eq:Delta-n}
\Delta_n
=
\frac1N\sum_{j=1}^N\E\mathfrak d_\alpha(U^{:,j},Z^{:,j}).
\end{equation}
Global smoothness implies
\[
\mathfrak d_\alpha(u,z)
\le(\alpha+L)\|u-z\|^2,
\]
so $\Delta_n<\infty$ whenever $\mathcal E_n<\infty$.  Under $\mathfrak d_\alpha\le\beta$, one has $\Delta_n\le\beta$.

We compare one full-drift Euler step with the stationary array.

\begin{lemma}\label{lem:Euler-comparison}
Assume $h\le\alpha/L^2$, and put
\[
\begin{aligned}
U_+^{i,j}
&=X_n^{i,j}-h\barV_i'(X_n^{i,j},q_{X_n}^{-i}),\\
Z_+^{i,j}
&=Y_{nh}^{i,\tau_{n,i}(j)}
-h\barV_i'(Y_{nh}^{i,\tau_{n,i}(j)},q_{Y_{nh}}^{-i}).
\end{aligned}
\]
Then
\begin{equation}\label{eq:Euler-comparison}
\frac1N\sum_{i,j}\E|U_+^{i,j}-Z_+^{i,j}|^2
\le
(1-\alpha h)\mathcal E_n+2h\Delta_n.
\end{equation}
\end{lemma}

\begin{proof}
Define the permuted Euler outputs
\[
\begin{aligned}
\widehat U_+^{i,j}
&=U^{i,j}-h\barV_i'(U^{i,j},q_{X_n}^{-i}),\\
\widehat Z_+^{i,j}
&=Z^{i,j}-h\barV_i'(Z^{i,j},q_{Y_{nh}}^{-i}).
\end{aligned}
\]
For each $i$, summation over $j$ is invariant under the permutation $\sigma_i$; hence
\[
\frac1N\sum_{i,j}\E|\widehat U_+^{i,j}-\widehat Z_+^{i,j}|^2
=
\frac1N\sum_{i,j}\E|U_+^{i,j}-Z_+^{i,j}|^2.
\]
Fix $i,j$, condition on the two arrays and on $\sigma_i$, and average only over $(\sigma_k)_{k\ne i}$.  Then $U^{i,j}$ and $Z^{i,j}$ are fixed, while $U^{-i,j}$ has law $q_{X_n}^{-i}$ and $Z^{-i,j}$ has law $q_{Y_{nh}}^{-i}$.  Therefore
\[
\begin{aligned}
\widehat U_+^{i,j}-\widehat Z_+^{i,j}
&=
\E\left[
U^{i,j}-h\partial_iV(U^{:,j})
-Z^{i,j}+h\partial_iV(Z^{:,j})
\,\bigm|\,
X_n,Y_{nh},\sigma_i
\right].
\end{aligned}
\]
Conditional Jensen's inequality (in the form $|\E[H\mid\mathcal F]|^2\le\E[|H|^2\mid\mathcal F]$), followed by summation over $i,j$, gives
\[
\begin{aligned}
\frac1N\sum_{i,j}\E|U_+^{i,j}-Z_+^{i,j}|^2
&\le
\frac1N\sum_{j=1}^N
\E\|U^{:,j}-h\nabla V(U^{:,j})
-Z^{:,j}+h\nabla V(Z^{:,j})\|^2.
\end{aligned}
\]
The pointwise Euler identity \eqref{eq:euler-defect-identity} yields
\[
\|u-h\nabla V(u)-z+h\nabla V(z)\|^2
\le(1-2\alpha h+L^2h^2)\|u-z\|^2
+2h\mathfrak d_\alpha(u,z).
\]
The step-size assumption makes the first coefficient at most $1-\alpha h$.  Taking expectations and averaging over $j$, the contraction term is
\[
(1-\alpha h)\frac1N\sum_{j=1}^N
\E\|U^{:,j}-Z^{:,j}\|^2
=(1-\alpha h)\mathcal E_n,
\]
where the equality follows from the rowwise optimal matchings.  Likewise, by definition \eqref{eq:Delta-n}, the defect contribution is
\[
2h\frac1N\sum_{j=1}^N
\E\mathfrak d_\alpha(U^{:,j},Z^{:,j})
=2h\Delta_n.
\]
Equation~\eqref{eq:Euler-comparison} follows.
\end{proof}

\subsection{Local diffusion remainder}

For every matched coordinate, define
\begin{equation*}
A^{i,j}
=
\int_{nh}^{(n+1)h}
\left[
\barV_i'(Y_t^{i,\tau_{n,i}(j)},q^{\star,-i})
-
\barV_i'(Y_{nh}^{i,\tau_{n,i}(j)},q^{\star,-i})
\right]\dd t.
\end{equation*}

The remaining local error comes from replacing the stationary diffusion by one Euler step.

\begin{lemma}\label{lem:local}
For every $(i,j)$,
\begin{equation}\label{eq:local-one}
\E|A^{i,j}|^2
\le
\frac23L^2h^4\E_{q^\star}|\partial_iV|^2+2L^2h^3.
\end{equation}
Consequently,
\begin{equation}\label{eq:local-sum}
\frac1N\sum_{i,j}\E|A^{i,j}|^2
\le
\frac23L^2h^4\Gamma_\star+2mL^2h^3.
\end{equation}
\end{lemma}

\begin{proof}
Write $\ell=\tau_{n,i}(j)$.  By \eqref{eq:conditional-smooth} and Jensen's inequality,
\[
\E|A^{i,j}|^2
\le L^2h\int_0^h
\E|Y_{nh+t}^{i,\ell}-Y_{nh}^{i,\ell}|^2\dd t.
\]
By the coupling rule \eqref{eq:Ycoupled}, the matched particle $Y^{i,\ell}$ is driven on this time interval by $B^{i,j}$, and hence
\[
Y_{nh+t}^{i,\ell}-Y_{nh}^{i,\ell}
=-\int_0^t\barV_i'(Y_{nh+s}^{i,\ell},q^{\star,-i})\dd s
+\sqrt2(B_{nh+t}^{i,j}-B_{nh}^{i,j}).
\]
By Cauchy--Schwarz in time,
\[
\E|Y_{nh+t}^{i,\ell}-Y_{nh}^{i,\ell}|^2
\le
2t\int_0^t
\E|\barV_i'(Y_{nh+s}^{i,\ell},q^{\star,-i})|^2\,\dd s
+4t.
\]
\autoref{lem:stationary-array} implies that $Y_{nh+s}^{i,\ell}\sim q^{\star,i}$ for every $s\in[0,h]$, not only at the grid times.  Conditional Jensen's inequality therefore gives
\[
\E|\barV_i'(Y_{nh+s}^{i,\ell},q^{\star,-i})|^2
\le\E_{q^\star}|\partial_iV|^2,
\qquad 0\le s\le h.
\]
Therefore
\[
\E|Y_{nh+t}^{i,\ell}-Y_{nh}^{i,\ell}|^2
\le2t^2\E_{q^\star}|\partial_iV|^2+4t.
\]
Integration over $t\in[0,h]$ proves \eqref{eq:local-one}; summing proves \eqref{eq:local-sum}.
\end{proof}

\subsection{One-step recursion}

The preceding estimates yield the one-step recursion.

The recursion uses the step-size threshold
\begin{equation}\label{eq:stepsize}
0<h\le h_0(B)
:=
\min\left\{
\frac{\alpha}{L^2},
\frac{\alpha B}{8L^2},
\frac1\alpha
\right\}.
\end{equation}
The first restriction is required by \autoref{lem:Euler-comparison}, the second
absorbs the batch variance through
\[
\frac{2h^2L^2}{B}\le\frac{\alpha h}{4},
\]
and the third controls the Young-inequality factors in the stationary-drift and
local-diffusion remainders.  The $B$-independent condition
$h\le\alpha/(8L^2)$ is only the worst-case choice $B=1$.  Under a finite
uniform defect one has $\alpha\le L$, so for sufficiently large $B$ the
restriction $h\le\alpha/L^2$ is the active one.

\begin{proposition}\label{prop:recursion}
Assume \eqref{eq:stepsize}. Then
\begin{equation}\label{eq:recursion}
\begin{aligned}
\mathcal E_{n+1}
&\le
\left(1-\frac{\alpha h}{4}\right)\mathcal E_n
+3h\Delta_n
+\frac{10h\kappa^2}{\alpha}\varepsilon_N^2(q^\star)\\
&\quad+
\frac{2h^2}{B}\Gamma_\star
+\frac{20L^2h^2}{\alpha}(h\Gamma_\star+m).
\end{aligned}
\end{equation}
\end{proposition}

\begin{proof}
For the matched pair $(X_n^{i,j},Y_{nh}^{i,\tau_{n,i}(j)})$, set
\[
\begin{aligned}
C^{i,j}&=U_+^{i,j}-Z_+^{i,j},\\
M^{i,j}&=h\left[
\barV_i'(Y_{nh}^{i,\tau_{n,i}(j)},q^{\star,-i})
-
\barV_i'(Y_{nh}^{i,\tau_{n,i}(j)},q_{Y_{nh}}^{-i})
\right].
\end{aligned}
\]
The Brownian increments in $\widetilde X_{n+1}^{i,j}$ and $Y_{(n+1)h}^{i,\tau_{n,i}(j)}$ cancel by construction.  Adding and subtracting the empirical stationary drift gives the exact decomposition
\begin{equation}\label{eq:exact-decomposition}
\widetilde X_{n+1}^{i,j}
-Y_{(n+1)h}^{i,\tau_{n,i}(j)}
=C^{i,j}+M^{i,j}+A^{i,j}.
\end{equation}
The term $M^{i,j}$ accounts for the fact that, conditional on $Y_{nh}$, the
stationary array has product empirical law $q_{Y_{nh}}$ rather than the
deterministic law $q^\star$.

For $\eta>0$,
\[
\|C+M+A\|^2
\le(1+\eta h)\|C\|^2
+2\left(1+\frac1{\eta h}\right)(\|M\|^2+\|A\|^2).
\]
Choose $\eta=\alpha/4$.  \autoref{lem:Euler-comparison} and $\alpha h\le1$ give
\[
\begin{aligned}
(1+\alpha h/4)
\left[(1-\alpha h)\mathcal E_n+2h\Delta_n\right]
&\le
\left(1-\frac{\alpha h}{2}\right)\mathcal E_n+3h\Delta_n.
\end{aligned}
\]
By \autoref{lem:measure-lip},
\[
\frac1N\sum_{i,j}|M^{i,j}|^2
\le h^2\kappa^2\cW^2(q_{Y_{nh}},q^\star).
\]
Taking expectations and using \autoref{lem:stationary-array} gives
\[
\frac1N\sum_{i,j}\E|M^{i,j}|^2
\le h^2\kappa^2\varepsilon_N^2(q^\star).
\]
Moreover, $2(1+4/(\alpha h))\le10/(\alpha h)$, and \autoref{lem:local} gives
\[
\frac{10}{\alpha h}\frac1N\sum_{i,j}\E|A^{i,j}|^2
\le
\frac{20L^2h^2}{\alpha}(h\Gamma_\star+m).
\]
Combining these estimates with \eqref{eq:exact-decomposition} gives the full
empirical-drift recursion with coefficient $1-\alpha h/2$.

For the finite-batch iterate, \autoref{lem:batch} makes the batch error
orthogonal to the full-drift comparison error:
\[
H^{i,j}:=\widetilde X_{n+1}^{i,j}
-Y_{(n+1)h}^{i,\tau_{n,i}(j)}
\]
is measurable with respect to $X_n$, $Y_{nh}$, and the Brownian increments on $[nh,(n+1)h]$.  Therefore \eqref{eq:batch-orthogonality} gives
\[
\E[E_n^{i,j}H^{i,j}]=0.
\]
Consequently,
\[
\begin{aligned}
&\frac1N\sum_{i,j}
\E|X_{n+1}^{i,j}-Y_{(n+1)h}^{i,\tau_{n,i}(j)}|^2=
\frac1N\sum_{i,j}
\E|\widetilde X_{n+1}^{i,j}-Y_{(n+1)h}^{i,\tau_{n,i}(j)}|^2
+
\frac1N\sum_{i,j}\E|E_n^{i,j}|^2.
\end{aligned}
\]
The last term is at most
\[
\frac{2h^2}{B}(L^2\mathcal E_n+\Gamma_\star).
\]
By the batch-dependent restriction $h\le\alpha B/(8L^2)$,
\[
\frac{2h^2L^2}{B}\le\frac{\alpha h}{4}.
\]
After this absorption, the contraction coefficient becomes
$1-\alpha h/4$, and \eqref{eq:recursion} follows.  Finally, the matched pairs at time $(n+1)h$ define an admissible rowwise coupling of $q_{X_{n+1}}$ and $q_{Y_{(n+1)h}}$. Optimizing can only decrease the cost, so the left-hand side bounds $\mathcal E_{n+1}$.
\end{proof}

\section{Main stability bounds}\label{sec:main-stability}

The main theorem retains the defects encountered by the stationary coupling.
The uniform-defect bound follows by summing the resulting geometric series.

\begin{theorem}\label{thm:main-defect}
Assume global smoothness and the standing initial-moment condition \eqref{eq:standing-initial-moment}.  Let $q^\star\in\cP_6(\R^m)$ be a product law satisfying the stationary fixed-point identities \eqref{eq:stationary-fixed-point}.  If \eqref{eq:stepsize} holds, define
\[
\rho=1-\frac{\alpha h}{4}
\]
and
\[
R_{h,N,B}
=
\frac{10h\kappa^2}{\alpha}\varepsilon_N^2(q^\star)
+\frac{2h^2}{B}\Gamma_\star
+\frac{20L^2h^2}{\alpha}(h\Gamma_\star+m).
\]
Then, for every $n\ge1$,
\begin{equation}\label{eq:main-defect}
\mathcal E_n
\le
\rho^n\mathcal E_0
+3h\sum_{k=0}^{n-1}\rho^{n-1-k}\Delta_k
+\frac{1-\rho^n}{1-\rho}R_{h,N,B}.
\end{equation}
Consequently,
\begin{equation}\label{eq:main-defect-W2}
\begin{aligned}
\left(\E \cW^2(q_{X_n},q^\star)\right)^{1/2}
&\le
\rho^{n/2}
\left(\E \cW^2(q_{X_0},q^\star)\right)^{1/2}
+2\varepsilon_N(q^\star)+
\left(3h\sum_{k=0}^{n-1}\rho^{n-1-k}\Delta_k\right)^{1/2}\\
&\quad+
\frac{\sqrt{40}\,\kappa}{\alpha}\varepsilon_N(q^\star)
+\sqrt{\frac{8h\Gamma_\star}{\alpha B}} +
\sqrt{\frac{80L^2h}{\alpha^2}(h\Gamma_\star+m)}.
\end{aligned}
\end{equation}
\end{theorem}

\begin{proof}
\autoref{prop:recursion} has the form
\[
\mathcal E_{n+1}\le\rho\mathcal E_n+3h\Delta_n+R_{h,N,B}.
\]
Repeated substitution gives
\[
\mathcal E_n
\le\rho^n\mathcal E_0
+3h\sum_{k=0}^{n-1}\rho^{n-1-k}\Delta_k
+R_{h,N,B}\sum_{k=0}^{n-1}\rho^k,
\]
which is \eqref{eq:main-defect}.

Since $1-\rho=\alpha h/4$, the stationary residual satisfies
\begin{equation*}
\begin{aligned}
\sqrt{\frac{R_{h,N,B}}{1-\rho}}
&=\sqrt{\frac{4}{\alpha h}R_{h,N,B}}\\
&\le
\frac{\sqrt{40}\,\kappa}{\alpha}\varepsilon_N(q^\star)
+\sqrt{\frac{8h\Gamma_\star}{\alpha B}}
+\sqrt{\frac{80L^2h}{\alpha^2}(h\Gamma_\star+m)}.
\end{aligned}
\end{equation*}
Using $\sqrt{a+b+c}\le\sqrt a+\sqrt b+\sqrt c$ again in
\eqref{eq:main-defect} yields
\[
\begin{aligned}
\sqrt{\mathcal E_n}
&\le\rho^{n/2}\sqrt{\mathcal E_0}
+
\left(3h\sum_{k=0}^{n-1}\rho^{n-1-k}\Delta_k\right)^{1/2} +
\frac{\sqrt{40}\,\kappa}{\alpha}\varepsilon_N(q^\star)
+\sqrt{\frac{8h\Gamma_\star}{\alpha B}}
+
\sqrt{\frac{80L^2h}{\alpha^2}(h\Gamma_\star+m)}.
\end{aligned}
\]
For each realization, the Wasserstein triangle inequality gives
\[
\cW(q_{X_n},q^\star)
\le \cW(q_{X_n},q_{Y_{nh}})+\cW(q_{Y_{nh}},q^\star).
\]
Minkowski's inequality in $L^2$, followed by \autoref{lem:stationary-array}, implies
\begin{equation}\label{eq:triangle-time-n}
\left(\E \cW^2(q_{X_n},q^\star)\right)^{1/2}
\le\sqrt{\mathcal E_n}+\varepsilon_N(q^\star).
\end{equation}
Similarly,
\begin{equation}\label{eq:triangle-time-zero}
\sqrt{\mathcal E_0}
\le
\left(\E \cW^2(q_{X_0},q^\star)\right)^{1/2}
+\varepsilon_N(q^\star).
\end{equation}
Substitution of \eqref{eq:triangle-time-zero} into \eqref{eq:triangle-time-n} gives \eqref{eq:main-defect-W2}; the factor $2\varepsilon_N$ uses $\rho^{n/2}\le1$.
\end{proof}

The accumulated-defect estimate can be localized to regions visited with high
probability.

\begin{remark}
\label{rem:localized-defect}
The accumulated-defect theorem does not require a global uniform defect
bound.  For any measurable set $G_n\subset\R^m\times\R^m$, the encountered
defect satisfies the exact decomposition
\begin{equation}
\label{eq:localized-defect}
\Delta_n
\leq
\sup_{(u,z)\in G_n}\defect_\alpha(u,z)
+
(\alpha+L)\,
\E\!\left[
\|U^{:,J}-Z^{:,J}\|^2
\mathbf 1_{\{(U^{:,J},Z^{:,J})\notin G_n\}}
\right],
\end{equation}
where $J$ is uniform on $[N]$ and independent of the construction.  Here we
used $\defect_\alpha(u,z)\leq(\alpha+L)\|u-z\|^2$.  Thus a local defect bound
on a region visited with high probability can be substituted into
\eqref{eq:main-defect}; one only needs a separate tail estimate for the
second term in \eqref{eq:localized-defect}.  Retaining $\Delta_n$ therefore permits local defect estimates supplemented by
a tail bound.
\end{remark}

Under a uniform defect bound, the accumulated term reduces to a geometric sum.

\begin{corollary}\label{cor:uniform}
Assume the uniform defect condition $\mathfrak d_\alpha\le\beta$, and let $q^\star$ be any MFVI minimizer.  Then
\begin{equation}\label{eq:main-uniform}
\begin{aligned}
\left(\E \cW^2(q_{X_n},q^\star)\right)^{1/2}
&\le
\rho^{n/2}
\left(\E \cW^2(q_{X_0},q^\star)\right)^{1/2}+
\left(2+\frac{\sqrt{40}\,\kappa}{\alpha}\right)
\varepsilon_N(q^\star)
+\sqrt{\frac{12\beta}{\alpha}}\\
&\quad+
\sqrt{\frac{8h\Gamma_\star}{\alpha B}}
+
\sqrt{\frac{80L^2h}{\alpha^2}(h\Gamma_\star+m)}.
\end{aligned}
\end{equation}
Using \eqref{eq:kappa-bound}, \eqref{eq:Gamma}, and \eqref{eq:epsN}, every term on the right-hand side is explicit in $(\alpha,\beta,L,m,h,N,B)$.
\end{corollary}

\begin{proof}
Since $\Delta_k\le\beta$,
\[
3h\sum_{k=0}^{n-1}\rho^{n-1-k}\Delta_k
\le
3h\beta\frac{1-\rho^n}{1-\rho}
\le\frac{12\beta}{\alpha}.
\]
The conclusion is derived by applying \autoref{thm:main-defect}.
\end{proof}

The terms in the uniform estimate have distinct sources.

\begin{remark}\label{rem:interpretation}
The noninitial terms in \eqref{eq:main-uniform} are product-empirical approximation, curvature defect, stochastic-gradient error, and Euler contributions from the local drift and Brownian increments.  The batch term vanishes as $B\to\infty$, while the displayed Euler term is $O(\sqrt h)$ for fixed model parameters.  When $\beta>0$, the theorem gives stability but does not imply exact convergence to a selected optimizer.  Nevertheless, \autoref{cor:minimizer-diameter} shows that all MFVI minimizers lie within $\sqrt{\beta/\alpha}$ of one another in $\cW$, so the nonuniqueness itself is quantitatively controlled by the same defect scale.
\end{remark}

For comparison, consider the strongly convex case.

\begin{remark}\label{rem:comparison-original}
If $\beta=0$, then $V$ is $\alpha$-strongly convex and the MFVI minimizer is unique.  Our general recursion gives the contraction factor
\[
(1-\alpha h/4)^{n/2},
\]
whereas \cite[Theorem~4.1]{DWZZ2026} obtains
$(1-\alpha h/2)^{n/2}$.  The weaker constant results from the Young inequalities used to absorb the
empirical stationary-drift discrepancy and the finite-batch term while
retaining nonzero encountered defects.  The proof in \cite{DWZZ2026} is specialized to the zero-defect setting and
retains a sharper contraction constant. The factor-of-two loss here affects
only numerical constants in the geometric transient and residual prefactors;
the orders in $N$, $B$, $h$, and the defect magnitude are unchanged.

There is also a difference in the particle term.  The general sixth-moment estimate used in \autoref{thm:main-defect} gives $\varepsilon_N=O(N^{-1/4})$.  Under $\beta=0$, strong log-concavity permits the sharper specialization
\[
\varepsilon_N(q^\star)
=O\!\left(\sqrt{\frac{m\log N}{\alpha N}}\right)
\]
from \autoref{rem:empirical-rates}.  These losses arise from using one argument for nonconvex marginals without
log-concavity; no optimality is claimed in the strongly convex subcase.
\end{remark}

The same estimates control the distance to the full minimizer set.

\begin{corollary}\label{cor:minimizer-set}
Under the assumptions of \autoref{cor:uniform}, let
\[
\mathcal M=
\operatorname*{argmin}_{q\in\cQ}\KL(q\|p),
\qquad
\operatorname{dist}_{\cW}(q,\mathcal M)
=\inf_{r\in\mathcal M}\cW(q,r),
\]
and define the uniform bounds
\[
\overline\Gamma
=
\frac{mL^2(1+\beta)}{\alpha},
\qquad
\overline\varepsilon_N
=
\frac{C\sqrt m}{\sqrt\alpha}
\bigl((1+\beta)(3+\beta)(5+\beta)\bigr)^{1/6}N^{-1/4},
\]
where $C$ is the universal constant in \eqref{eq:epsN}.  Then
\begin{equation}\label{eq:minimizer-set-bound}
\begin{aligned}
\left(\E\operatorname{dist}_{\cW}^2(q_{X_n},\mathcal M)\right)^{1/2}
&\le
\rho^{n/2}
\inf_{q^\star\in\mathcal M}
\left(\E \cW^2(q_{X_0},q^\star)\right)^{1/2}+
\left(2+\frac{\sqrt{40}\,\kappa}{\alpha}\right)\overline\varepsilon_N
+\sqrt{\frac{12\beta}{\alpha}}\\
&\quad+
\sqrt{\frac{8h\overline\Gamma}{\alpha B}}
+
\sqrt{\frac{80L^2h}{\alpha^2}(h\overline\Gamma+m)}.
\end{aligned}
\end{equation}
\end{corollary}

\begin{proof}
We first note that $\mathcal M$ is closed in $\cW$.  If
$q_n\in\mathcal M$ and $q_n\to q$ in $\cW$, then
\autoref{lem:product-closed} shows that $q$ is a product measure, while lower
semicontinuity of relative entropy gives
\[
\KL(q\|p)\le\liminf_{n\to\infty}\KL(q_n\|p)
=\inf_{r\in\cQ}\KL(r\|p).
\]
Thus $q$ is a minimizer.  Its relative entropy is finite, so
$q\ll p\ll\mathcal L^m$ and its marginals are absolutely continuous; hence
$q\in\mathcal M$.  The distance to $\mathcal M$ is therefore measurable.

For every $q^\star\in\mathcal M$, \autoref{prop:fixed} gives the same
coordinatewise fixed-point identities.  The bounds in \autoref{lem:moments},
\eqref{eq:Gamma}, and \eqref{eq:epsN} depend only on
$(\alpha,\beta,L,m)$, so
\[
\Gamma_\star\le\overline\Gamma,
\qquad
\varepsilon_N(q^\star)\le\overline\varepsilon_N
\]
uniformly over $\mathcal M$.  Apply \autoref{cor:uniform} to a fixed
$q^\star\in\mathcal M$, use
$\operatorname{dist}_{\cW}(q,\mathcal M)\le\cW(q,q^\star)$, and take
the infimum over $q^\star$ to obtain \eqref{eq:minimizer-set-bound}.
\end{proof}

\section{An arbitrary-dimensional nonconvex benchmark with known MFVI minimizer}\label{sec:benchmark}

The benchmark uses the following Pinsker-type uniqueness criterion, which does
not require global strong convexity.

\begin{proposition}
\label{prop:bounded-feature-uniqueness}
For $i\in[m]$, let
$\pi_i(\dd x)=Z_i^{-1}e^{-\nu_i(x)}\dd x$ be a probability law and let $s_i:\R\to\R$ be bounded measurable functions satisfying
\[
\int s_i\,\dd\pi_i=0,
\qquad
\|s_i\|_\infty\leq M_i,\qquad M_i>0.
\]
Let $J=J^{\mathsf T}$ have zero diagonal and consider
\[
V(x)
=
\sum_{i=1}^m\nu_i(x_i)
+
\frac12 s(x)^{\mathsf T}Js(x),
\qquad
s(x)=(s_1(x_1),\ldots,s_m(x_m))^{\mathsf T}.
\]
Set $D=\operatorname{diag}(M_1,\ldots,M_m)$.  If
\begin{equation*}
\|DJD\|_{\mathrm{op}}<1,
\end{equation*}
then $\pi_1\otimes\cdots\otimes\pi_m$ is the unique MFVI minimizer.
\end{proposition}

\begin{proof}
For $q=\otimes_iq^i$, put
$K_i=\KL(q^i\|\pi_i)$ and
$r_i=\int s_i\,\dd q^i$.  Up to an additive constant independent of $q$,
the MFVI objective equals
\[
\sum_{i=1}^mK_i+\frac12r^{\mathsf T}Jr.
\]
Pinsker's inequality and the centering of $s_i$ give
$|r_i|\leq M_i\sqrt{2K_i}$.  Writing $r=Dz$ gives
$\|z\|^2\leq2\sum_iK_i$, and therefore
\[
\frac12r^{\mathsf T}Jr
\geq
-\frac12\|DJD\|_{\mathrm{op}}\|z\|^2
\geq
-\|DJD\|_{\mathrm{op}}\sum_iK_i.
\]
The objective gap is therefore at least
$(1-\|DJD\|_{\mathrm{op}})\sum_iK_i$, which is nonnegative and vanishes
only when every $q^i=\pi_i$.
\end{proof}

Let
\[
\nu_a(x)=\frac{x^2}{2}+a\cos x,
\qquad a>1,
\]
and let
\[
\pi_a(\dd x)=Z_a^{-1}e^{-\nu_a(x)}\dd x.
\]
The density $\pi_a$ is even.  Let $J\in\R^{m\times m}$ be symmetric with zero diagonal, and set
\[
s(x)=(\tanh x_1,\ldots,\tanh x_m)^\mathsf T.
\]
Consider
\begin{equation}\label{eq:benchmark-potential}
V_{a,J}(x)
=
\sum_{i=1}^m\nu_a(x_i)
+\frac12s(x)^\mathsf TJ s(x).
\end{equation}
Since $\nu_a''(0)=1-a<0$, the family is nonconvex. Its gradient has at most
linear growth and its Hessian is globally bounded.

For a matrix $J$, write
\[
\|J\|_\infty=\max_i\sum_k|J_{ik}|.
\]

The MFVI minimizer of this benchmark is explicit and unique.

\begin{proposition}\label{prop:benchmark}
Assume $a>1$, $J=J^\mathsf T$, $J_{ii}=0$, and $\|J\|_{\mathrm{op}}<1$.  Then the unique MFVI minimizer associated with \eqref{eq:benchmark-potential} is
\begin{equation*}
q^\star=\pi_a^{\otimes m}.
\end{equation*}
Moreover, \autoref{ass:smooth} holds with
\begin{equation}\label{eq:benchmark-L}
L\le1+a+\|J\|_{\mathrm{op}}+2\|J\|_\infty,
\end{equation}
and, for every $\alpha\in(0,1)$, the uniform defect bound holds with
\begin{equation}\label{eq:benchmark-beta}
\beta
=
\frac{m(a+\|J\|_{\mathrm{op}})^2}{1-\alpha}.
\end{equation}
For this model, the projected-drift measure-sensitivity constant satisfies
\begin{equation*}
\kappa\le\|J\|_{\mathrm{op}}.
\end{equation*}
\end{proposition}

\begin{proof}
Let $q=\otimes_iq^i$ have finite MFVI objective, and define
\[
K_i=\KL(q^i\|\pi_a),
\qquad
r_i=\int\tanh x\,q^i(\dd x),
\qquad r=(r_1,\ldots,r_m)^\mathsf T.
\]
Since $\pi_a$ is even, $\int\tanh x\,\pi_a(\dd x)=0$.  Up to a constant independent of $q$, the MFVI objective is
\begin{equation*}
\sum_{i=1}^mK_i+\frac12r^\mathsf T Jr.
\end{equation*}
Pinsker's inequality and $\|\tanh\|_\infty\le1$ give
\[
|r_i|
=
\left|\int\tanh x\,(q^i-\pi_a)(\dd x)\right|
\le\sqrt{2K_i}.
\]
Therefore
\[
\frac12r^\mathsf TJr
\ge-\frac12\|J\|_{\mathrm{op}}\|r\|^2
\ge-\|J\|_{\mathrm{op}}\sum_iK_i.
\]
The objective gap is bounded below by
\[
(1-\|J\|_{\mathrm{op}})\sum_iK_i,
\]
which is nonnegative, with equality only when every $q^i=\pi_a$. Hence the
MFVI minimizer is unique.

Let $D(x)=\operatorname{diag}(\operatorname{sech}^2x_i)$.  The interaction gradient is
\[
\nabla\left(\frac12s^\mathsf TJs\right)=DJs.
\]
Its Hessian is
\[
DJD+
\operatorname{diag}\left(-2\operatorname{sech}^2(x_i)\tanh(x_i)(Js)_i\right).
\]
The first term has operator norm at most $\|J\|_{\mathrm{op}}$.  Since $|(Js)_i|\le\|J\|_\infty$, the diagonal term has norm at most $2\|J\|_\infty$.  The separable Hessian has norm at most $1+a$, proving \eqref{eq:benchmark-L}.

Write $V_{a,J}(x)=\|x\|^2/2+U(x)$.  Since
\[
\nabla U(x)=-a\sin x+DJs,
\]
we have
\[
\|\nabla U(x)\|
\le a\sqrt m+\|J\|_{\mathrm{op}}\|s(x)\|
\le\sqrt m(a+\|J\|_{\mathrm{op}}).
\]
\autoref{prop:bounded-perturbation} with $\alpha_0=1$ gives \eqref{eq:benchmark-beta}.

Finally, only cross derivatives enter the matrix $A_z(u)$ in \eqref{eq:hybrid-matrix}.  The separable part contributes nothing, and the cross-derivative matrix is $D_1JD_2$ for diagonal contractions $D_1,D_2$ with norms at most one.  Hence $\kappa\le\|J\|_{\mathrm{op}}$.
\end{proof}

For the numerical experiments we use a sparse ring interaction.

\begin{remark}
For the numerical experiments, $J$ is a nearest-neighbor ring matrix rescaled so that $\|J\|_{\mathrm{op}}=\lambda<1$.  Its row-sum norm is also bounded independently of $m$.  Hence the interaction contribution to $L$ and the sensitivity constant
$\kappa$ remain dimension-independent, whereas the conservative uniform
defect \eqref{eq:benchmark-beta} grows linearly with $m$.
\end{remark}

\section{Numerical validation}\label{sec:numerics}

All code used to generate the numerical results in this section is available
at the project repository
\href{https://github.com/tvu25/MFVI}{\texttt{github.com/tvu25/MFVI}}.
Further implementation details and the organization of the numerical scripts
are described in Appendix~\ref{app:code}.

The experiments use \eqref{eq:benchmark-potential}.  Unless otherwise stated,
$m=2$, $a=2$, and
\[
J=\begin{pmatrix}0&1/2\\1/2&0\end{pmatrix}.
\]
The exact MFVI minimizer is $q^\star=\pi_2\otimes\pi_2$.  For these baseline parameters,
\eqref{eq:benchmark-L} gives $L\leq4.5$.  With $\alpha=3/4$, the refined
step-size condition \eqref{eq:stepsize} gives $h_0(1)\geq0.00462$ for the
smallest batch $B=1$, while for $B\geq8$ the deterministic bound
$\alpha/L^2\geq0.0370$ is active.  Hence the baseline choice $h=0.004$
satisfies the proved restriction even for $B=1$; the coarse/fine step-size
study up to $h=0.008$ also lies below the deterministic full-drift bound.
The projected drifts are
\[
\barV_1'(x,q^2)
=x-2\sin x+\frac12\operatorname{sech}^2(x)
\int\tanh y\,q^2(\dd y),
\]
and similarly for coordinate two.  Both the full product-empirical drift and its finite-batch approximation can
therefore be evaluated without numerical differentiation.

\subsection{Protocol}

The density and quantiles of $\pi_a$ are computed by deterministic quadrature on $[-8,8]$.  For an empirical row with $N$ particles, its one-dimensional $\cW$ error is evaluated by sorting and comparing with the midpoint quantiles of $\pi_a$.  Product errors use the additivity of squared $\cW$.  Curves labeled ``root mean-square error'' report
\[
\left(\frac1R\sum_{r=1}^R \cW^2(q_{X_n^{(r)}},q^\star)\right)^{1/2},
\]
and shaded regions or error bars equal approximately two standard errors obtained by the delta method.  All experiments generated by \texttt{expanded\_experiments.py} use the base NumPy seed \texttt{20260717}; study-specific seeds are deterministic offsets recorded in the source.  The complete source code, all seeds, raw CSV tables, and plotting scripts are publicly available in the companion GitHub repository \cite{NV2026}; see Appendix~\ref{app:code}.

Two initialization regimes are used.  The convergence and defect experiments start each particle near $3$, which probes a strongly asymmetric initial condition in the nonconvex landscape.  The $N$-, $h$-, $B$-, and dimension-scaling experiments start from independent samples of $q^\star$; this isolates product-empirical, batch, and discretization effects from the much larger transient caused by the asymmetric initialization.

\begin{table}[t]
\centering
\small
\setlength{\tabcolsep}{4pt}
\caption{Baseline numerical configurations.  ``Full'' denotes the exact projected drift under the current product empirical law, corresponding to the $B\to\infty$ limit.}
\label{tab:numerical-config}
\begin{tabular}{@{}lllll@{}}
\toprule
Study & Initialization & $N$ & $h$ & Other parameters\\
\midrule
Time convergence & $\mathcal N(3,0.35^2)$ & 384 & 0.004 & $B=1,4,16,\mathrm{Full}$, $T=6$\\
Particle scaling & stationary & 64--768 & 0.002 & $B=16$, $T=2$\\
Step-size scaling & coupled stationary arrays & 1024 & 0.001--0.008 & Full drift, $h_{\rm ref}=5\!\times\!10^{-4}$\\
Batch scaling & stationary arrays & 2048 & -- & Drift RMSE, $B=1,\ldots,256$\\
Defect scaling & $\mathcal N(3,0.35^2)$ & 384 & 0.002 & $a=1.05,1.6,2.2,3$, $T=6$\\
Dimension scaling & stationary & 256 & 0.002 & $m=2,4,8,16$, ring $J$\\
\bottomrule
\end{tabular}
\end{table}

\subsection{Finite-batch, full-drift, and CAVI comparisons}

\autoref{fig:method-comparison} compares the actual finite-batch PAVI update with the full product-empirical drift.  From the asymmetric initialization, all PAVI variants reduce the root mean-square error from approximately $4.73$ to between $0.77$ and $0.83$ by time $6$.  The small separation between $B=1,4,16$ and the full drift indicates that, in this example, the nonconvex transient and empirical approximation dominate the batch error.  This behavior is consistent with the decomposition in \autoref{cor:uniform}
and should not be interpreted as evidence that the batch term is absent in
general.

For reference, we also implement sequential CAVI by deterministic quadrature on the same one-dimensional grid.  At every sweep, the normalized quadrature weights define a discrete probability measure supported on the grid nodes, and its one-dimensional Wasserstein error to $\pi_a$ is evaluated through the quantile identity
\[
\cW^2(\mu,\nu)=\int_0^1|F_\mu^{-1}(t)-F_\nu^{-1}(t)|^2\,\dd t.
\]
The product error is then obtained by additivity.  It reaches $\cW$ error $2.3\times10^{-3}$ after four sweeps and $3.1\times10^{-4}$ after five.  This comparison is specific to the benchmark: each CAVI update can be
normalized accurately on a one-dimensional grid, whereas PAVI is intended for
settings in which conditional normalization is unavailable or expensive.  The
comparison is therefore diagnostic rather than a runtime ranking.

\begin{figure}[t]
\centering
\begin{subfigure}[t]{0.49\textwidth}
\centering
\includegraphics[width=\textwidth]{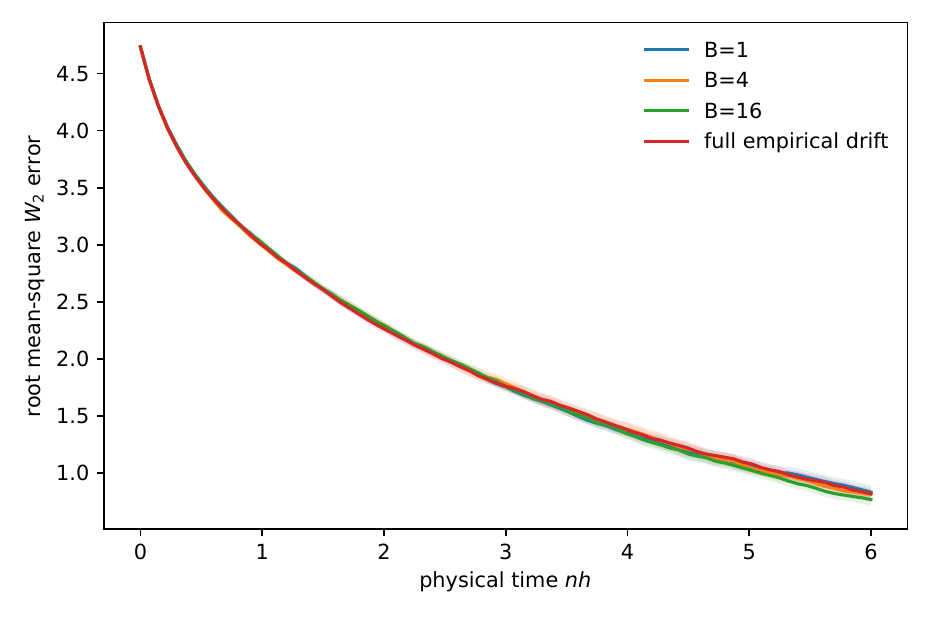}
\caption{Finite-batch and full empirical-drift PAVI.}
\end{subfigure}
\hfill
\begin{subfigure}[t]{0.49\textwidth}
\centering
\includegraphics[width=\textwidth]{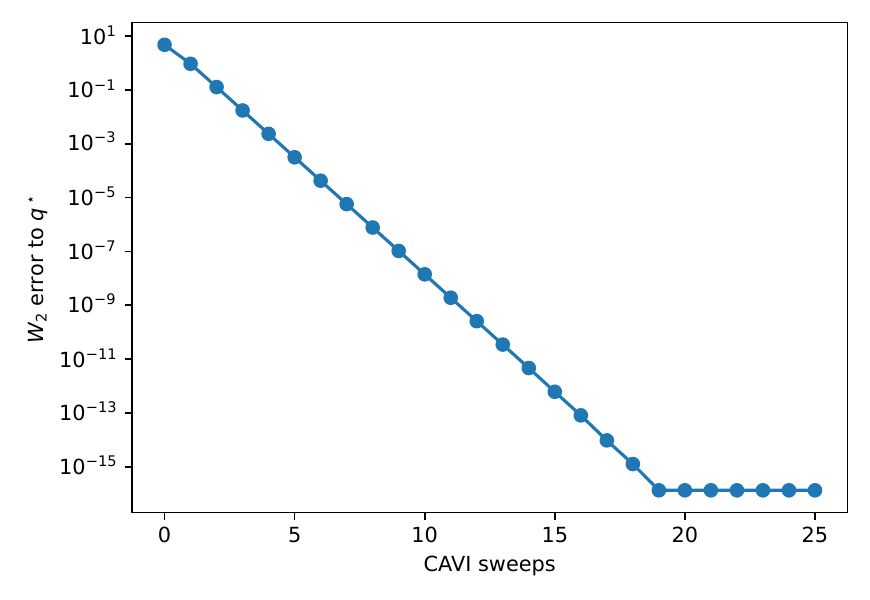}
\caption{Grid-based sequential CAVI; $\cW$ is computed by the one-dimensional quantile formula.}
\end{subfigure}
\caption{Method comparisons on the exact nonconvex benchmark.  The horizontal variables are physical time for PAVI and coordinate sweeps for CAVI, so the panels are not presented as a direct wall-clock comparison.}
\label{fig:method-comparison}
\end{figure}

\subsection{Particle and step-size dependence}

\autoref{fig:particle-step} reports terminal errors from stationary
initialization.  The fitted slope against particle number is $-0.458$, close to
the $-1/2$ scaling observed for this smooth benchmark and faster than the
worst-case $N^{-1/4}$ root-mean-square bound obtained from sixth moments alone.
To isolate time discretization from the empirical floor, the full-drift Euler
scheme at step size $h$ is coupled to finer schemes using the same initial array
and nested Brownian increments.  With reference steps $5\times10^{-4}$ and
$2.5\times10^{-4}$, the fitted slopes over the same four coarse step sizes are
$1.110$ and $1.010$, respectively.  The RMS distance between the two reference
paths is approximately $2.56\times10^{-4}$, so the coarser reference is treated
as a high-resolution comparator rather than a converged solution.  Both fits
are consistent with first-order strong error for Euler--Maruyama with additive
noise under sufficient smoothness; see \cite{KP1992}.  The $O(\sqrt h)$ term in
\autoref{cor:uniform} is therefore conservative for this benchmark.

\begin{figure}[t]
\centering
\begin{subfigure}[t]{0.49\textwidth}
\centering
\includegraphics[width=\textwidth]{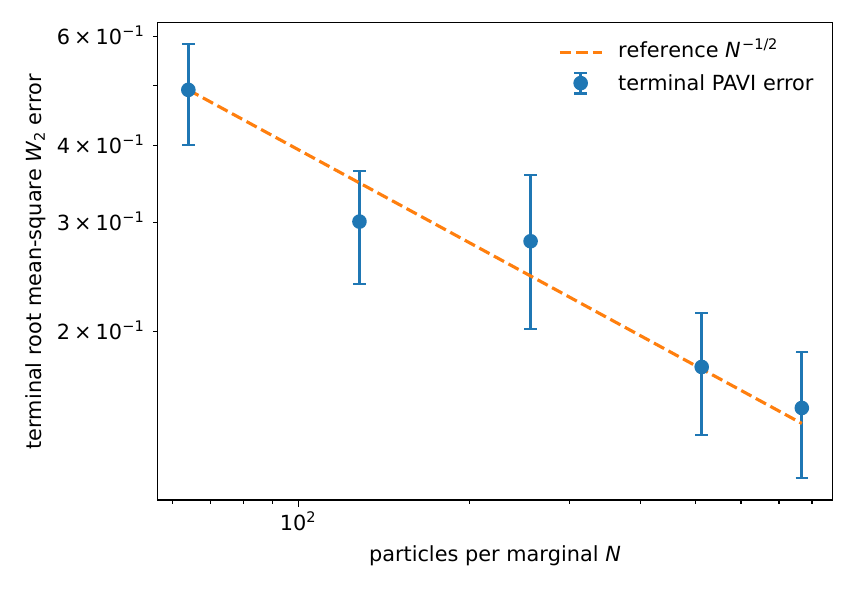}
\caption{Terminal PAVI error versus $N$.}
\end{subfigure}
\hfill
\begin{subfigure}[t]{0.49\textwidth}
\centering
\includegraphics[width=\textwidth]{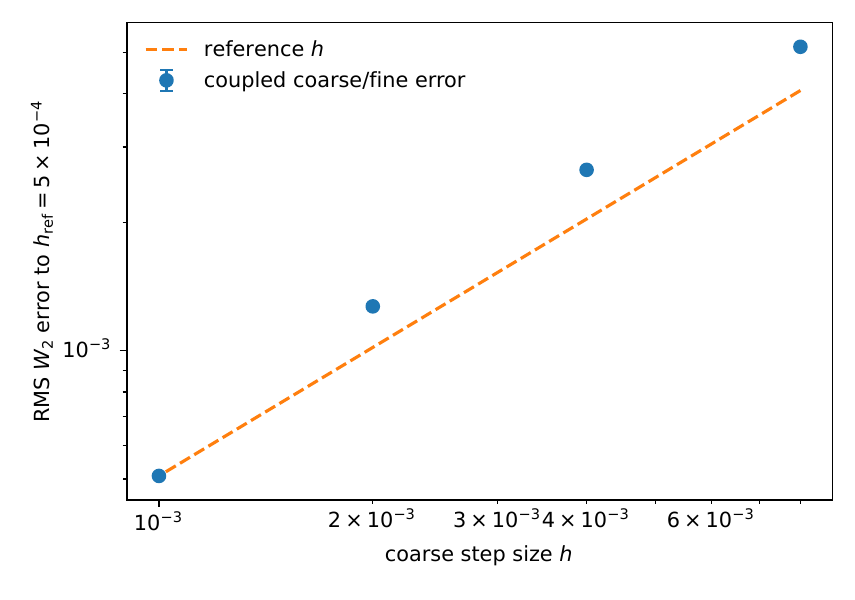}
\caption{Coupled Euler error versus $h$.}
\end{subfigure}
\caption{Particle scaling and a coupled coarse/fine time-discretization study.}
\label{fig:particle-step}
\end{figure}

\subsection{Batch size and curvature defect}

The left panel of \autoref{fig:batch-defect} measures the finite-batch projected-drift error relative to the full product-empirical drift at stationary particle arrays.  The fitted log--log slope is $-0.505$, matching the $B^{-1/2}$ root-mean-square scaling predicted by \autoref{lem:batch}.  When terminal target error is dominated by product-empirical approximation,
this direct diagnostic isolates the batch effect more effectively.  The right panel varies $a$, keeping $\lambda=1/2$ and $\alpha=3/4$.  \autoref{prop:benchmark} gives the conservative bound
\[
\beta=8(a+1/2)^2.
\]
As this bound increases from $19.22$ to $98$, the terminal error from the asymmetric initialization increases from $0.190$ to $2.513$.  The bound $\beta$ is conservative, but the observed monotone deterioration is
consistent with the defect term in the stability estimate.

\begin{figure}[t]
\centering
\begin{subfigure}[t]{0.49\textwidth}
\centering
\includegraphics[width=\textwidth]{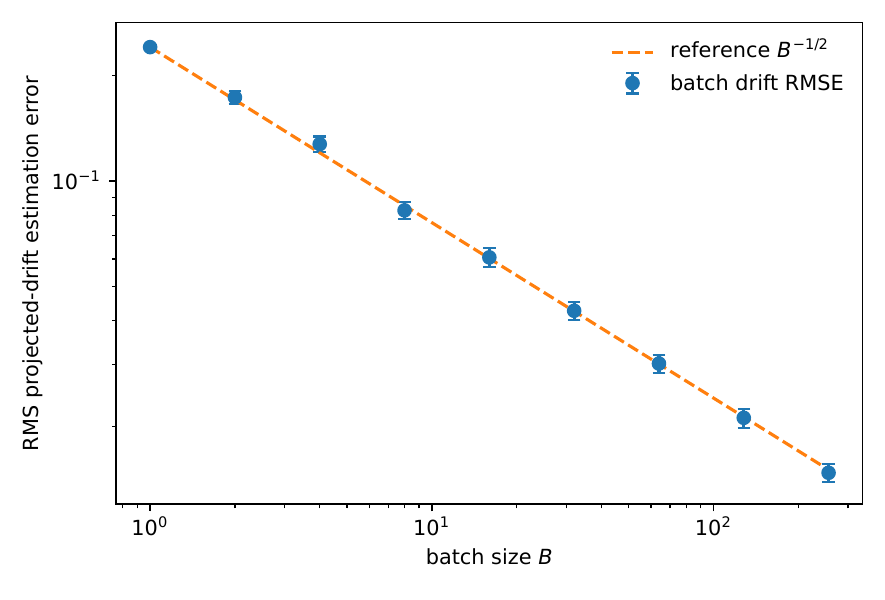}
\caption{Finite-batch drift error versus $B$.}
\end{subfigure}
\hfill
\begin{subfigure}[t]{0.49\textwidth}
\centering
\includegraphics[width=\textwidth]{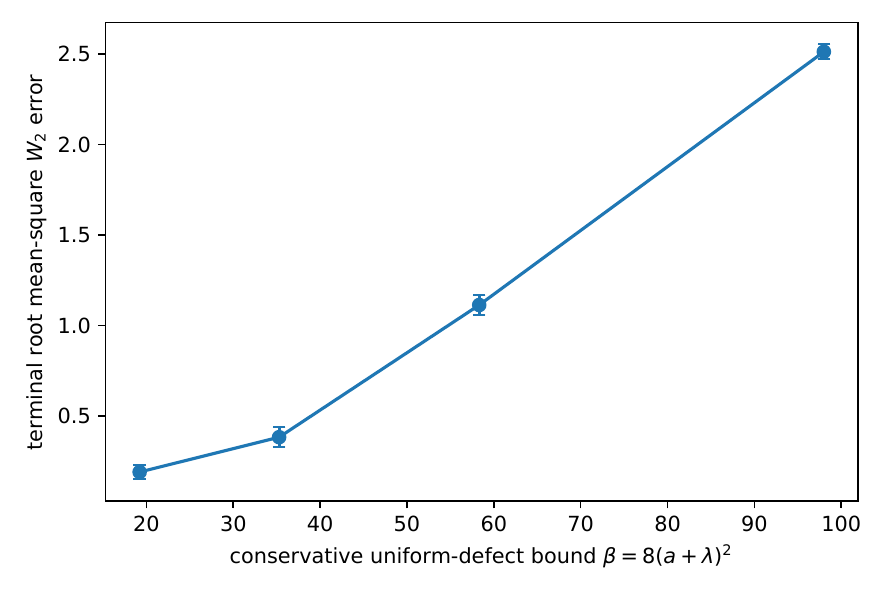}
\caption{Dependence on the uniform-defect bound.}
\end{subfigure}
\caption{Direct stochastic-gradient and curvature-defect diagnostics.}
\label{fig:batch-defect}
\end{figure}

\subsection{Dimension and computational cost}

For the ring family, the terminal error rises from $0.251$ at $m=2$ to $0.639$ at $m=16$, broadly consistent with the $\sqrt m$-type dependence of the product-empirical and Euler terms. After normalization by $\sqrt m$, the terminal error remains nearly
constant over the tested dimensions, in agreement with the
per-coordinate interpretation in \autoref{rem:coordinate-defects}. The generic drift benchmark evaluates all $mNB$ hybrid coordinate gradients rather than using the sufficient-statistic reduction available for \eqref{eq:benchmark-potential}.  In the reported reproducibility run, a log--log fit over workloads at least $32768$ gives slope $0.952$, consistent with the linear work count in \autoref{prop:complexity}.  The measured times are implementation- and hardware-dependent; the reproducible point is the scaling with the number of coordinate-gradient evaluations.

\begin{figure}[t]
\centering
\begin{subfigure}[t]{0.49\textwidth}
\centering
\includegraphics[width=\textwidth]{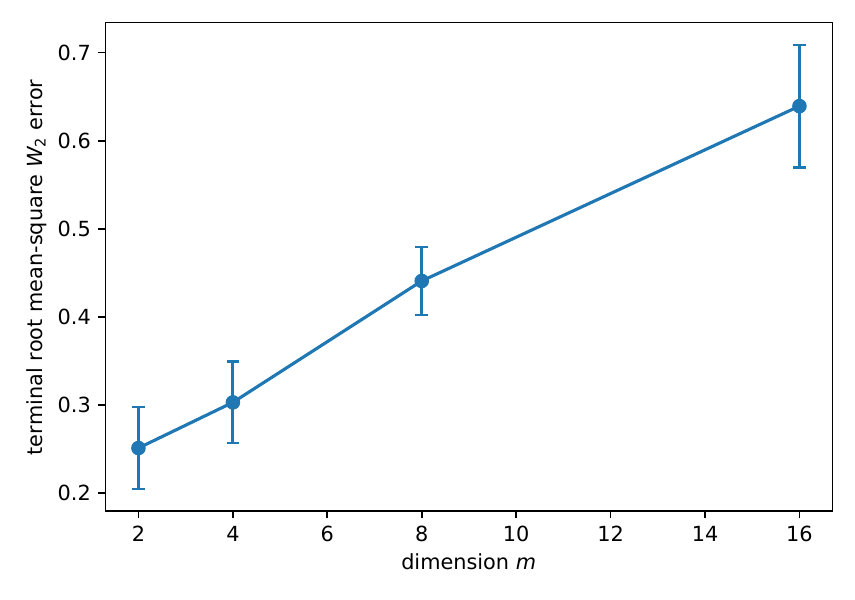}
\caption{Terminal error versus dimension.}
\end{subfigure}
\hfill
\begin{subfigure}[t]{0.49\textwidth}
\centering
\includegraphics[width=\textwidth]{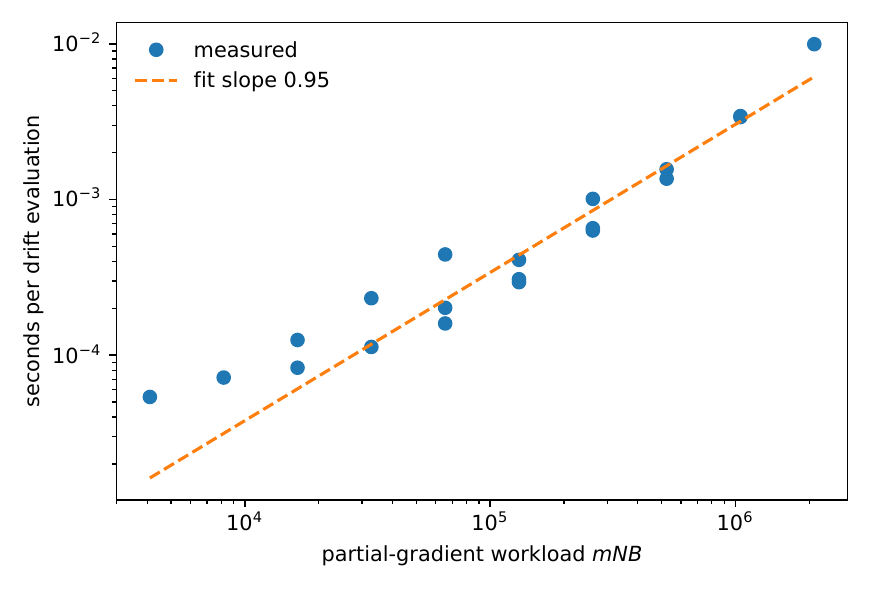}
\caption{Generic drift-evaluation cost.}
\end{subfigure}
\caption{Dimension and computational scaling.}
\label{fig:dimension-cost}
\end{figure}

\paragraph{Numerical conclusions.}
The experiments reproduce the predicted $B^{-1/2}$ batch scaling and show
nearly $N^{-1/2}$ empirical scaling for this smooth benchmark.  The coupled
time-discretization study is approximately first order, and the observed
dimension dependence is consistent with the product geometry.  Increasing the
nonconvexity parameter also increases the terminal error from the asymmetric
initialization.  These experiments do not determine sharp constants or the
optimal dependence on $\beta$.  Because the MFVI minimizer is known
analytically, the reported errors are not contaminated by an unknown
optimization error; the remaining reference error comes only from
one-dimensional quadrature and truncation.

\section{Why polynomial drift growth requires another algorithm}\label{sec:polynomial}

Polynomial Hessian growth cannot replace global smoothness while the explicit
update remains untamed.  Consider
\[
V(x)=\frac{x^4}{4}+\frac{x^2}{2}.
\]
The potential is globally strongly convex, whereas the one-particle Euler
update is
\[
X_{n+1}
=
X_n-h(X_n^3+X_n)+\sqrt{2h}\,\xi_n.
\]
Conditionally on $X_n=x$,
\[
\E[|X_{n+1}|^2\mid X_n=x]
=
|x-h(x^3+x)|^2+2h,
\]
which grows like $h^2|x|^6$.  Hence no global quadratic Foster--Lyapunov inequality can hold, in agreement
with the classical explicit-Euler instability for superlinear drifts
\cite{HJK2011,HJK2012}.

Implicit and semi-implicit schemes, including backward Euler--Maruyama, and tamed explicit methods can handle polynomial drift growth in single-particle SDEs; see \cite{BDMS2019,HJK2012}.  Extending these ideas to PAVI requires controlling the interaction between the stabilization mechanism and the product-empirical projected drift.  In particular, directly taming the stochastic batch estimator $g_n^i(x)$ generally changes its conditional mean and may destroy the unbiasedness
\[
\E[g_n^i(x)\mid X_n]=\barV_i'(x,q_{X_n}^{-i}),
\]
which is used in the orthogonality argument of \autoref{lem:batch}.  Implicit treatment of the projected drift creates a different computational issue because each coordinate solve depends on a random empirical law.  
Related tamed interacting-particle Langevin schemes for superlinear
drifts have been analyzed in \cite{JMS2025}.  Their setting differs
from PAVI, where the projected drift depends on a random product
empirical law and the finite-batch unbiasedness is used essentially
in \autoref{lem:batch}.  Extending the present PAVI analysis to a
suitable tamed or implicit discretization therefore requires
additional ideas and is left for future work.

\section{Discussion}\label{sec:discussion}

The curvature-defect formulation separates the use of strong convexity from
the global smoothness needed by the explicit discretization.  When the coupling
rarely visits nonconvex regions, \eqref{eq:main-defect} is sharper than the
uniform-defect corollary.  Two limitations of the present argument are
structural:
\begin{enumerate}[label=(\roman*)]
\item a uniform additive defect alone does not imply uniqueness of the MFVI minimizer or exact contraction to a selected minimizer;
\item explicit PAVI is not stable for general superlinear drifts.
\end{enumerate}
Reflection-coupling methods for dissipative or curvature-at-infinity drifts
\cite{E2016,EGZ2019,MMS2020} address a complementary regime.  For exact
Langevin diffusions, such methods can yield strict contraction without an
additive residual in a concave or otherwise modified transport metric even
when Euclidean synchronous coupling is not contractive at short distances. In
contrast, the present argument retains the standard quadratic Wasserstein
geometry and is built around synchronous coupling because that coupling is
compatible with the coordinatewise product projection, the stationary
empirical array, and the finite-batch orthogonality. The resulting estimate
contains an explicit accumulated-defect term but gives a direct non-asymptotic
decomposition of particle, batch, and Euler errors for the implementable
scheme.  Extending reflection coupling to the random product-empirical drift would
require a joint coupling of the state and the empirical law approximation.
A possible extension is to combine curvature-at-infinity assumptions with a
coupling adapted to the random product-empirical drift. A second direction is
a structure-preserving tamed or implicit PAVI scheme for polynomial
potentials.

\appendix

\section{Auxiliary Wasserstein facts}

We collect two standard Wasserstein facts used above.  The first is additivity
for product measures.

\begin{lemma}\label{lem:product-additivity}
For $\mu_i,\nu_i\in\cP_2(\R^{d_i})$,
\[
\cW^2\left(\bigotimes_{i=1}^m\mu_i,
\bigotimes_{i=1}^m\nu_i\right)
=
\sum_{i=1}^m\cW^2(\mu_i,\nu_i).
\]
\end{lemma}

\begin{proof}
The product of optimal marginal couplings gives the upper bound.  Every coupling of the product measures induces a coupling of each marginal pair, which gives the reverse inequality.
\end{proof}

The second is the $L^2$ triangle inequality for random Wasserstein distances.

\begin{lemma}\label{lem:random-triangle}
For random measures $\mu,\nu,\rho\in\cP_2(\R^d)$,
\[
\left(\E\cW^2(\mu,\nu)\right)^{1/2}
\le
\left(\E\cW^2(\mu,\rho)\right)^{1/2}
+
\left(\E\cW^2(\rho,\nu)\right)^{1/2}.
\]
\end{lemma}

\begin{proof}
Apply the triangle inequality for $\cW$ pointwise and Minkowski's inequality in $L^2$.
\end{proof}

\section{Reproducibility and implementation details}
\label{app:code}
The complete source code used for the numerical experiments is publicly available in the companion GitHub repository \cite{NV2026}, \url{https://github.com/tvu25/MFVI}.
The repository contains the following three principal scripts.

The file \texttt{pavi.py} provides a generic NumPy implementation of
Algorithm~\ref{alg:pavi}.  In particular, each sample from the current product
empirical law is obtained by drawing the coordinate indices independently.
This distinction is important: using a common particle index across all
coordinates would sample from the empirical law of the particle columns rather
than from the product of the marginal empirical measures.

The file \texttt{expanded\_experiments.py} implements the benchmark family
studied in \autoref{sec:numerics} and generates the numerical figures and
CSV tables reported there.  The random seeds and the numerical parameters used
in the experiments are fixed in the script.

The file \texttt{check\_stepsize\_halving.py} performs the additional
reference-step verification used in the time-discretization study.  It couples
the Euler schemes with step sizes
\[
    h_{\mathrm{ref}}=5\times 10^{-4}
    \qquad\text{and}\qquad
    h_{\mathrm{fine}}=2.5\times 10^{-4}
\]
through nested Brownian increments, so that the comparison isolates the
effect of time discretization.

\smallskip
{\bf Data availability.}
No external dataset is used in this work.  The numerical data underlying the figures are generated by the supplied scripts, and the corresponding raw CSV tables are included in the companion repository \cite{NV2026}.

\end{document}